\documentclass[pdflatex,sn-mathphys-num]{sn-jnl}% Math and Physical Sciences Numbered Reference Style
\usepackage{amsfonts}
\usepackage{amsmath, amssymb, amsthm}
\usepackage{enumerate}
\usepackage{epsfig}
\usepackage{verbatim}
\usepackage{bm} % for $\bm{1}$ to give a blackboard $1$
\usepackage{manyfoot}
\usepackage{xcolor}
\usepackage{hyperref}

\makeatletter
\@ifundefined{SetFootnoteHook}{\newcommand{\SetFootnoteHook}[1]{}}{}
\makeatother

\usepackage[section]{algorithm} % for algorithm environment
\usepackage{algpseudocode} % for algorithmic environment
\algnewcommand\algorithmicinput{\textbf{Input:}}
\algnewcommand\Input{\item[\algorithmicinput]}
\algnewcommand\algorithmicoutput{\textbf{Output:}}
\algnewcommand\Output{\item[\algorithmicoutput]}

\newtheorem{theorem}{Theorem}[section]
\newtheorem{proposition}[theorem]{Proposition}
\newtheorem{lemma}[theorem]{Lemma}
\newtheorem{corollary}[theorem]{Corollary}
\newtheorem{definition}[theorem]{Definition}
\newtheorem{remark}[theorem]{Remark}

\numberwithin{equation}{section}

\allowdisplaybreaks[1] %allows page breaks within align environments

\providecommand{\e}{\varepsilon}

\providecommand{\R}{\ensuremath{\mathbb{R}}}
\providecommand{\N}{\ensuremath{\mathbb{N}}}

\newcommand{\bsvarrho}{{\boldsymbol \varrho}}

\newcommand{\realiz}{\mathrm{R}}

\newcommand{\depth}{L}
\newcommand{\size}{M}

\newcommand{\Parallel}[1]{{\rm P}(#1)} % Parallelization

\DeclareMathOperator{\ReLU}{ReLU}

\newcommand{\norm}[2][]{\| #2 \|_{#1}} % norm
\newcommand{\snorm}[2][]{| #2 |_{#1}} % seminorm
\begin{document}

\title[NNs for SPPs]{Neural Networks for Singular Perturbations -- Finite Regularity}

\author[1]{\fnm{F.} \sur{Rohner}}\email{fabian.rohner@sam.math.ethz.ch}
\author[1]{\fnm{Ch.} \sur{Schwab}}\email{christoph.schwab@sam.math.ethz.ch}
\author*[2]{\fnm{C.} \sur{Xenophontos}}\email{xenophontos.christos@ucy.ac.cy}

%\equalcont{These authors contributed equally to this work.}

\affil[1]{\orgdiv{Seminar for Applied Mathematics}, \orgname{ETH Z\"{u}rich}, \orgaddress{\street{R\"{a}mistrasse 101}, \city{Z\"urich}, \postcode{8092}, \country{Switzerland}}}

%\state{Z\"urich}, 

\affil*[2]{\orgdiv{Department of Mathematics and Statistics}, \orgname{University of Cyprus}, \orgaddress{\street{PO BOX 20537}, \city{Nicosia}, \postcode{1678},  \country{Cyprus}}}

%\state{Nicosia},

\abstract{
%{\small
We study finite-element and 
deep feedforward neural network (DNN for short) 
expressivity rate bounds for solution sets of
a model linear, second order singularly perturbed, 
elliptic two-point boundary value problem, 
in Sobolev norms on a bounded interval $(-1,1)$,
with explicit dependence on 
the singular perturbation parameter $\e\in (0,1]$.
Emphasis is on low Sobolev regularity of the data,
i.e., source term $f$ and reaction coefficient $b$.
A proof of $\e$-explicit solution regularity
based on exponentially weighted energy-norm bounds 
is developed, and
\emph{$\e$-robust, algebraic expression rate bounds} 
in Sobolev norms for $\mathbb{P}_1$ Finite-Elements 
on exponential and Shishkin type meshes
is proved.
Expression rates for 
shallow (fixed depth) $\ReLU$-NNs are shown
which are robust w.r. to $\e$ and explicit
in terms of the NN size.
Robust NN expression rate bounds are further studied 
for deep feedforward DNNs with ReLU and tanh-activations. 
As in \cite{OSX24_1085}, 
tanh- and sigmoid-activated sub-NNs allow to include
exponential boundary layer functions
exactly into the NN feature space, 
leading to reduced NN sizes.
Recent 
bitstring encoding techniques for 
deep NNs with ReLU activations afford, 
still under low data regularity $f,b \in H^1(I)$ 
\emph{twice the (robust) convergence rate of $\mathbb{P}_1$ Finite-Elements} 
achievable with ``eXp'' or Shishkin meshes. 
%}
}



\maketitle

%\textbf{Subject Classification:}
%34B08, % Parameter dependent boundary value problems for ordinary differential equations
%34D15, % Singular perturbations of ordinary differential equations
%65L11, % Numerical solution of singularly perturbed problems involving ordinary differential equations
%65N30 % Finite elements, Rayleigh-Ritz and Galerkin methods
%%%%%%%%%%%%%%%%%%%%%%%%%%%%%%%%%%%%%%%%%%%%%%%%%%%%%%%%%%%
\section{Introduction}
\label{sec:intro} 
%%%%%%%%%%%%%%%%%%%%%%%%%%%%%%%%%%%%%%%%%%%%%%%%%%%%%%%%%%%
Elliptic and parabolic PDEs which depend on a small, 
non-dimensional physical parameter $\e>0$ arise in many
areas of science. Their numerical solution is well-known
to be challenging due to multiscale phenomena governing
solution regularity. A particular class of such problems
are \emph{singular perturbation problems} in solid and
fluid dynamics; we refer to \cite{JLLSgPert,TemamSgPert,Linss1985,RST2nd}
and the references there.

In all these applications, PDEs depend on a small 
parameter $\e \in (0,1]$ in physically relevant regimes of input data.
Standard numerical approximation methods (Finite Volume-, 
Finite Difference- or Finite Element Methods) 
generally do \emph{not} perform uniformly 
w.r. to the physical perturbation parameter $\e$:
in general, a so-called \emph{scale resolution} condition
relating the discretization parameters (such as the mesh width
$h$ in Finite Element Methods) and $\e$ needs to hold,
in order to exhibit theoretical approximation rates.
In this paper, 
we continue and complement our study \cite{OSX24_1085} 
on $\e$-robust expression rate bounds for 
solution families $\{ u_\e \}_{0<\e\leq 1}$ 
of elliptic singular perturbation problems. 
As in \cite{OSX24_1085,SS1996}, 
we consider the 
model linear, singularly perturbed, 
reaction-diffusion boundary value problem (BVP): 
find $u_{\e}(x)$ such that
\begin{eqnarray}
\mathcal{L}_{\e} u_{\e} := -\e^2 u''_{\e}(x) + b(x)u_{\e}(x) &=& f(x) \; , \; x \in I = (-1,1) \; , 
\label{eq:de} 
\\
u_{\e}(\pm 1) &=& 0, 
\label{eq:bc}
\end{eqnarray}
where $\e \in (0,1]$ is a small parameter that can approach zero, 
and $b(x), f(x)$ are given 
functions on $\overline{I}=[-1,1]$, 
with the reaction coefficient satisfying
\begin{equation}\label{eq:bbounds}
\overline{b} \ge b(x) \ge \underline{b} > 0 \; \mbox{a.e. on} \;  \overline{I}
\;\mbox{ for some constants} \; \overline{b}, \;\underline{b}.
\end{equation}
While in \cite{OSX24_1085}, \emph{analytic data} $b,f \in C^\infty(\overline{I})$ 
were assumed, we concentrate in the present note on \emph{data $b,f$ of finite, low Sobolev
regularity} in $I$.

A common trait of singularly perturbed, 
\emph{linear, elliptic} PDEs is an 
additive decomposition of the solution $u_{\e}$
into a regular part $u^S_\e$ and into boundary layer components $u^{BL}_{\e}$
(see, e.g., \cite{TemamSgPert} and the references there).
The regular solution part $u^S_\e$ 
may depend on the singular perturbation parameter $\e$, 
but is smooth independent of $\e$, i.e.
derivatives of $u^S_\e$
satisfy bounds in Sobolev norms which are bounded uniformly 
in terms of $\e$. 
A typical result for 
solutions $u_\e$ of \eqref{eq:de}, \eqref{eq:bc} 
is as follows
(see, e.g. \cite{TemamSgPert,RST2nd,Linss1985} and the references there).
Assume that for a nonnegative integer $k$, 
the data in \eqref{eq:de} satisfy $b, f \in C^{k+2}(\bar{I})$.
Then 
\begin{equation}\label{eq:uDec}
u_\e = u^S_\e + u^{BL}_\e \;,
\end{equation}
where the smooth part $u^S_\e$ 
and the boundary layer $u^{BL}_\e$ 
satisfy:
for integer $j=0,1,2,...,k+2$ 
exists $C_j > 0$ such that, 
for $\e\in (0 , 1]$ and for $x\in I$ 
holds
\begin{equation}\label{eq:SBLsmooth}
|(u^S_\e)^{(j)}(x)| \leq C_j \;,
\quad 
|(u^{BL}_\e)^{(j)}(x)| \leq C_j \e^{-j} \exp(-({\rm dist}(x,\partial I))/\e) \;.
\end{equation}
The solution part $u^{BL}_\e$
is not uniformly smooth in terms of $\e$. 
Its $k$-th derivative typically grows as $O(\e^{-k})$.
Due to their exponential decay off $\partial I$, the
$u^{BL}_\e$ are referred to as \emph{boundary layers}.
%%%%%%%%%%%%%%%%%%%%%%%%%%%%%%%%%%%%%%%%%%%%%%%%%%%%%%%%%%%
\subsection{Contributions}
\label{sec:Contrib}
%%%%%%%%%%%%%%%%%%%%%%%%%%%%%%%%%%%%%%%%%%%%%%%%%%%%%%%%%%%
While the pointwise bounds \eqref{eq:SBLsmooth} are, in terms
of their $\e$-dependence, optimal, one observes that 
at least $C^2(\overline{I})$-regularity of the data $b$ and $f$ 
in \eqref{eq:de}, \eqref{eq:bc} is required for them to hold.

We establish in Prop.~\ref{prop:Reg1} 
for a model, linear elliptic singular perturbation problem 
with source term $f\in H^1(I)$ and variable reaction coefficient $b\in H^1(I)$ 
in a bounded interval $I = (-1,1)$ an $\e$-explicit 
solution decomposition \eqref{eq:uDec}
into a regular part $u^S_\e$ of low, finite Sobolev regularity bounded
uniformly with respect to the perturbation parameter $\e$, 
and 
into exponential boundary layer functions $u^{BL}_\e$.
Unlike \cite{SS1996,Melenk1997} where analytic in $\overline{I}$ 
coefficients and source terms were assumed, we admit data of
finite Sobolev regularity $H^1(I)$. Unlike decompositions with 
pointwise derivative bounds \eqref{eq:SBLsmooth}, we establish $\e$-explicit
bounds on $u^{BL}_\e$ in exponentially weighted Sobolev spaces.

Based on the solution decomposition, 
we prove that the solution can be expressed robustly 
(i.e. with constants and rates independent of $\e$)
at algebraic (w.r. to the NN size as measured in terms of the number
of neurons) rates by various adapted Finite-Element Methods and DNN architectures, 
in various Sobolev norms. 
We show, in particular, in Thm.~\ref{thm:relubalanced}
that fixed depth, strict $\ReLU$-activated NNs
afford the same robust algebraic rates as $\mathbb{P}_1$-FEM on 
$\e$-dependent Shishkin, Bakhvalov and exponential meshes, 
in the ($\e$-dependent) energy norm.

For singular perturbations of 
constant coefficient differential operators, 
we exhibit novel, fixed-depth NN architectures
that are based on combination of $\ReLU$ and $\tanh$ activations
which allow robust solution approximation rates
with fewer neurons than strict $\ReLU$-activated NNs.

We prove that certain deep, so-called 
\emph{super-expressive} NN architectures from 
\cite{yang2025deepneuralnetworksgeneral}
afford, \emph{under the same finite Sobolev regularity assumptions 
on the coefficient $b$ and the source terms $f$} in \eqref{eq:de},
higher orders of robust solution approximation than 
the best results achievable by shallow NNs:
Specifically, 
with a given budget of $N$ neurons
we show there are deep NNs that 
afford a robust convergence rate which is higher than the 
robust rate afforded with CpwL spline approximations.
%%%%%%%%%%%%%%%%%%%%%%%%%%%%%%%%%%%%%%%%%%%%%%%%%%%%%%%%%%%
\subsection{Notation}
\label{eq:Notat}
%%%%%%%%%%%%%%%%%%%%%%%%%%%%%%%%%%%%%%%%%%%%%%%%%%%%%%%%%%%
We shall use standard notation: 
$\N = \{1,2,3,...\}$ and $\R$ denote the natural resp. the real numbers, 
and $I=(a,b)\subset \R$ shall denote an open, bounded interval.
For a summability index $q\in [1,\infty]$, 
we denote by $L^q(I)$ the Banach space of $q$-Lebesgue integrable,
real-valued functions on $I$. 
For integer differentiation order $k\in \N_0 = \{0,1,2,...\}$,
we denote by $W^{k,q}(I)$ the Banach space of $k$-times weakly 
differentiable functions in $L^q(I)$, whose $k$-th weak derivative
belongs to $L^q(I)$. 
For $k\in \N$, $W^{k,q}_0(I)$ denotes the closed subspace of 
functions in $W^{k,q}(I)$ with vanishing traces at $\partial I$.
In the hilbertian case $q=2$, we write $H^k(I) = W^{k,2}(I)$,
and $H^k_0(I)$.
Often we shall use the space ${\mathcal S}^p(I,\Delta)$ of 
continuous, real-valued functions $f:I\to \R$ 
on a bounded interval $I=(a,b)$ which are piecewise polynomials
of degree $p\geq 1$ on a finite partition 
$\Delta = \{a, x_1, x_2,..., x_{N-1}, b\} $, $N \in \N$,
where the points $x_j$ are pairwise distinct, and enumerated in 
increasing order, with $x_0 = a$ and $x_N = b$.
%%%%%%%%%%%%%%%%%%%%%%%%%%%%%%%%%%%%%%%%%%%%%%%%%%%%%%%%%%%
\subsection{Layout}
\label{sec:Layout}
%%%%%%%%%%%%%%%%%%%%%%%%%%%%%%%%%%%%%%%%%%%%%%%%%%%%%%%%%%%
In Section~\ref{sec:model}, we present the model problem,
and recap several results on $\e$-explicit solution regularity.
In Sec.~\ref{sec:FinReg} we establish a solution decomposition
into a regular part which belongs to $H^2(I)$ independent of
$\e \in (0,1]$ and boundary layers $u^{BL}_\e$ which are 
bounded in an exponentially weighted $H^2(I)$ norm 
uniformly w.r. to $\e$.

Sec.~\ref{sec:approx} re-proves, under this regularity,
results on robust FE approximation 
by continuous piecewise linear functions 
which are (under stronger, $C^2$ regularity)
due to Shishkin, Bakhvalov and others.

Sec.~\ref{sec:relunn} develops novel results related to deep neural networks.
Sec.~\ref{sec:nn} recaps terminology and definitions from neural networks.
Sec.~\ref{sec:ReLU CpwL} recaps a (basically known) fact on the equivalence
of Cpwl functions of a single variable with certain shallow ReLU NNs.
This implies, with the results from Sec.~\ref{sec:approx}, 
in Sec.~\ref{sec:RobNNSolAppr} 
robust algebraic approximation rates in the energy norm by 
fixed depth, $\ReLU$-activated NNs,
under in a sense minimal, 
finite low Sobolev regularity of coefficients and source
terms in the differential equation.
Sec.~\ref{sec:SupReLU} then studies the benefit of depth: 
still under the same finite regularity assumptions as in Sec.~\ref{sec:approx},
\emph{deep} feedforward DNNs with a 
combination of $\ReLU$ and $\tanh$ activations 
afford robust solution approximation, \emph{at a rate
twice as high as that of the $\mathbb{P}_1$-FE methods
on Shishkin or exponential meshes}.
%%%%%%%%%%%%%%%%%%%%%%%%%%%%%%%%%%%%%%%%%%%%%%%%%%%%%%%%%%%
\section{Model Problem and its Regularity}
\label{sec:model}
%%%%%%%%%%%%%%%%%%%%%%%%%%%%%%%%%%%%%%%%%%%%%%%%%%%%%%%%%%%
%%%%%%%%%%%%%%%%%%%%%%%%%%%%%%%%%%%%%%%%%%%%%%%%%%%%%%%%%%%%%5
\subsection{Variational Formulation. A-priori Estimates}
\label{sec:ModPrbVar}
%%%%%%%%%%%%%%%%%%%%%%%%%%%%%%%%%%%%%%%%%%%%%%%%%%%%%%%%%%%%%5
The variational formulation of \eqref{eq:de}, \eqref{eq:bc} reads: 
given $f\in H^{-1}(I)$, $b\in L^\infty(I)$, 
for $\e\in (0,1]$
find $u_{\e}\in H^1_0(I)$ such that
\begin{equation}
\label{eq:BVPweak}
B_{\e}(u_{\e},v) = F(v) \quad \forall v\in H^1_0(I)\;.
\end{equation}
Here, for $v,w\in H^1(I)$, 
\[
B_{\e}(w,v) = \int_I \{\e^2 w'v' + b(x) wv \} dx,
\quad
F(v) = \int_I f(x) v dx 
\;.
\]
The Lax-Milgram Lemma ensures, 
for given data $f$ and $b$ in \eqref{eq:BVPweak} as above
with \eqref{eq:bbounds},
existence and uniqueness
of weak solutions $\{ u_{\e} \}_{0<\e\leq 1} \subset H^1_0(I)$ 
of \eqref{eq:BVPweak}: 
with the ``energy norm''  $\| \circ \|_{1,\e,I}$, defined
for $w\in H^1(I)$ by
\begin{equation}
\label{eq:energy}
\| w \|_{1,\e,I}^2  
: = \int_I \left( \e^2 w'(x)^2 + b(x) w(x)^2 \right)dx
\simeq 
\left( \e \| w' \|_{L^2(I)} + \| \sqrt{b}w \|_{L^2(I)} \right)^2
\end{equation}
(with constant implied in $\simeq$ independent of $\e$) 
there hold continuity and coercivity, uniformly w.r. to $\e\in (0,1]$:
\begin{equation}\label{eq:ContCoerc}
\forall v,w \in H^1_0(I): 
\quad 
| B_{\e}(w,v) | \leq \| v \|_{1,\e,I} \| w \|_{1,\e,I} \;,
\quad 
B_{\e}(v,v) \geq \| v \|_{1,\e,I}^2 \;.
\end{equation}
Assuming that $b\in L^\infty(I)$ with \eqref{eq:bbounds} and also that $f\in H^{-1}(I)$, 
existence and uniqueness of the parametric solutions 
$\{ u_{\e} \}_{0<\e\leq 1} \subset H^1_0(I)$ follow.

We shall require $\e$-explicit a-priori estimates on $u_{\e}$. 
%[
For $0<\e\leq 1$ and $f\in H^{-1}(I)$, 
we define the $\e$-dependent $H^{-1}$ norm
\begin{equation}\label{eq:H-1edef}
\|f \|_{-1,\e,I} := \sup_{v\in H^1_0(I)} \frac{|f(v)|}{\| v \|_{1,\e,I}} \;.
\end{equation}
By \cite[Prop.8.14]{Brezis},
with $(\cdot,\cdot)$ denoting the $L^2(I)$ inner product,
\begin{equation}\label{eq:fH-1rep}
\forall f\in H^{-1}(I) \exists g_0,g_1\in L^2(I): 
\;
\forall v\in H^1_0(I): f(v) = (g_0,v) + (g_1,v') \;.
\end{equation}
Then, with $\simeq$ and $\lesssim$ 
denoting (in)equality with hidden constants independent of $\e \in (0,1]$,
using \eqref{eq:fH-1rep}
\begin{equation}\label{eq:fH-1:2}
\begin{array}{rcl}
\|f \|_{-1,\e,I} &=& \displaystyle 
 \sup_{v\in H^1_0(I)} \frac{|(g_0,v) + (g_1,v')|}{\| v \|_{1,\e,I}} 
\simeq \sup_{v\in H^1_0(I)} \frac{|(g_0,v) + (g_1,v')|}{ \e \| v' \|_{L^2(I)}  + \| v \|_{L^2(I)}} 
\\
& \lesssim & \| g_0(f) \|_{L^2(I)} + \e^{-1} \| g_1(f) \|_{L^2(I)} \;.
\end{array}
\end{equation}
From \eqref{eq:ContCoerc} and \eqref{eq:H-1edef}
\[
\| u_{\e} \|_{1,\e,I}^2 \leq B(u_{\e},u_{\e}) = |f(u_{\e})|  \leq \| f \|_{-1,\e,I} \| u_{\e} \|_{1,\e,I} 
\;.
\]
This implies with \eqref{eq:fH-1:2} that for $0<\e \leq 1$ it holds
\begin{equation}\label{eq:apriori-1}
\| u_{\e} \|_{1,\e,I} \leq \| f \|_{-1,\e,I} \lesssim \| g_0(f) \|_{L^2(I)} + \e^{-1} \| g_1(f) \|_{L^2(I)} \;.
\end{equation}
Assume now in addition that $f\in L^2(I)$. 
Then it holds
$$
\forall v\in H^1_0(I): \quad 
|F(v)| 
\leq \frac{\| f \|_{L^2(I)}}{\sqrt{\underline{b}}} \| \sqrt{b} v \|_{L^2(I)}
\leq \frac{\| f \|_{L^2(I)}}{\sqrt{\underline{b}}} \| v \|_{1,\e,I}
$$
so that
\begin{equation}\label{eq:apriori0}
\| u_{\e} \|_{1,\e,I} \leq \frac{\| f \|_{L^2(I)}}{\sqrt{\underline{b}}}\;,
\quad 
0<\e \leq 1 \;.
\end{equation}
%%%%%%%%%%%%%%%%%%%%%%%%%%%%%%%%%%%%%%%%%%%%%%%%%%%%%%%%%%
\subsection{Regularity}
\label{sec:FinReg}
%%%%%%%%%%%%%%%%%%%%%%%%%%%%%%%%%%%%%%%%%%%%%%%%%%%%%%%%%%
In order to study $\e$-uniform approximation rates  
for finite-parametric approximations of the family 
$\{ u_{\e} \}_{0<\e\leq 1} \subset H^1_0(I)$,
higher regularity of $u_{\e}$ in Sobolev scales with explicit
dependence on $\e$ is required. 
In the following result, we establish a decomposition \eqref{eq:uDec} 
for the parametric weak solutions $\{ u_\e \}_{0<\e\leq 1}\subset H^1_0(I)$ 
of \eqref{eq:BVPweak}.
While results of this type are, basically,
well known (e.g. \cite{Linss1985,RST2nd,TemamSgPert} and references there), 
we provide a proof based on weighted energy bounds which
does not resort to maximum principles and barrier-function comparison arguments. 
It does provide sufficient structural information 
on $\{ u_\e \}_{0<\e\leq 1}\subset H^1_0(I)$ for the ensuing
proof of $\e$-uniform approximation rate bounds of FE methods
and of DNNs, under somewhat weaker regularity assumptions on
the data $f$ and $b$ in \eqref{eq:BVPweak} as in \cite{Linss1985,RST2nd}.
For all $k\in \N$, 
$H^k(I)\subset C^{k-1,1/2}(\overline{I})$ with continuous embedding (\cite[Chap.~8]{Brezis}).
\begin{proposition}\label{prop:Reg1}
Consider the boundary value problem \eqref{eq:BVPweak},
with $b,f \in H^k(I)$, for a $k \in \{1,2\}$,
and with \eqref{eq:bbounds}, i.e. 
$\overline{b} \ge b(x) \ge \underline{b} > 0$ 
for every $x \in \overline{I}$ (observe that $H^1(I)\subset C(\overline{I})$).

Then the 
weak solutions $\{ u_{\e} \}_{0<\e\leq 1} \subset H^1_0(I)$
of \eqref{eq:BVPweak} admit the decomposition
\begin{equation}
\label{eq:decomp}
u_{\e} = u_0 + u^{BL}_{\e,-} + u^{BL}_{\e,+} + u^R_\e.
\end{equation}
Here, 
in each case $k=1,2$, $u_0$ and $u^R_\e$ belong to $H^k(I)$ with 
\begin{equation}\label{eq:u0H2}
\| u_0 \|_{H^k(I)} \leq C , \;\; \| u^R_\e \|_{1,\e,I} \leq C\e^k 
\;.
\end{equation}
for $C(b,f)>0$ 
depending on $\|f\|_{H^k(I)}$, $\| b \|_{H^k(I)}$ 
but independent of $\e\in (0,1]$.

The 
\emph{boundary layer functions} $u^{BL}_{\e,\pm}$ 
in \eqref{eq:decomp} satisfy 
\emph{exponentially weighted a-priori estimates}: 
for $0<\e \leq 1$ and for $k=1,2$ holds
\begin{equation}\label{eq:uBLw}
\e^2 \| (u^{BL}_{\e,\pm})'' \|_{L^2(I;w_{\e}^\pm)} \leq  C(b,f) <\infty
\end{equation}
with the \emph{exponential weight functions}
\begin{equation}
\label{eq:weight}
w_{\e}^\pm(x) := \exp \left(\frac{1\mp x}{\e \beta_*}\right) \;,\quad x\in I.
\end{equation}
Here, $\beta_* = 2/\sqrt{\underline{b}} > 0$ is independent of $\e$
and 
the exponentially weighted $L^2$-norm is defined as 
$$
\| v \|^2_{L^2(I;w_{\e}^\pm)}
:=
\int_I v(x)^2 w_{\e}^\pm(x)^2 dx \;.
$$
Whenever statements hold w.r. to either weight function, 
we write 
$\| v \|_{L^2(I;w_{\e})}$.
\end{proposition}
\begin{proof}
Observe that the assumption $b,f \in H^k(I)$ for $k\in \{1,2\}$
implies with Sobolev's embedding that there are unique modifications
$b,f \in C^{k-1/2}(\overline{I})$.

We proceed in several steps:

\noindent
Step 1 [regularity of $u_0$]: 
The \emph{smooth} part $u_0$ in \eqref{eq:decomp}
is the $L^2(I)$-limit, as $\e \to 0$, of $u_{\e}$. I.e.
\begin{equation}
\label{u0}
u_0 = f/b.
\end{equation}
We calculate $u_0' = f'/b - fb^{-2} b'$, and 
using \eqref{eq:bbounds} have $u_0' \in L^2(I)$ in case $k=1$, 
i.e. $u_0\in H^1(I)$ for $b,f\in H^1(I)$.

Assume now $k=2$. 
Using again \eqref{eq:bbounds}, we have
$$ 
u''_0 = b^{-3}(f''b^2 - 2f'b'b - f b''b + 2f(b')^2) \in L^2(I)
$$
which shows that $u_0 \in H^2(I)$ under the provision that $b, f \in H^2(I)$.

\noindent
Step 2 [splitting definition]:
The boundary layer functions $u^{BL}_{\e,\pm}$ in \eqref{eq:decomp}
satisfy 
\begin{equation}
\left.\left.\begin{array}{rl}
\mathcal{L}_{\e} u^{BL}_{\e, -} & =0 \quad \text { in } I 
\\
u^{BL}_{\e,-}(-1) & =- u_0(-1) \\ u^{BL}_{\e,-}(1)&=0
\end{array}\right\} \; \begin{array}{rl}
\mathcal{L}_{\e} u^{BL}_{\e,+}& =0 \quad \text { in } I 
\\
u^{BL}_{\e,+}(-1) & =0 
\\
u^{BL}_{\e,+}(1)&=-u_0(1)
\end{array}\right\}. 
\label{BLs}
\end{equation}
As $u_0 = f/b \in H^1(I) \subset C(\overline{I})$, 
the nonhomogeneous Dirichlet boundary conditions in \eqref{BLs} are well-defined,
and for $k=1,2$ and $0<\e\leq 1$, 
there exist unique solutions $u^{BL}_{\e,\pm}\in H^{k+1}(I)$ of \eqref{BLs}.

\noindent
Step 3 [regularity of $u^R_\e$]: \newline
Due to \eqref{eq:decomp},
the \emph{remainder} $u^R_\e$ satisfies 
(using $\mathcal{L}_{\e} (u^{BL}_{\e,\pm}) = 0$)
\begin{equation}
\label{rem}
\begin{array}{rcl}
\mathcal{L}_{\e} u^R_{\e} &=& 
\mathcal{L}_{\e} (u_{\e} - u^{BL}_{\e,+} - u^{BL}_{\e,-} - u_0) =
\e^2 u''_{0} \quad \mbox{in}\;\; I  \; , 
\\
u^R_{\e}(\pm1) & = & (u_{\e} - u^{BL}_{\e,+} - u^{BL}_{\e,-} - u_0) (\pm 1) 
\\
 & = & - (u^{BL}_{\e,+} + u^{BL}_{\e,-} + u_0)(\pm 1)  = 0 \;.
\end{array}
\end{equation}

Assume $k=1$. 
Step 1 implies then $u_0 \in H^1(I)$ so that $u_0'' \in H^{-1}(I)$.
Using \eqref{eq:apriori-1}
with $g_0 = 0$ and with $g_1(f) = \e^2 u_0' \in L^2(I)$,
we find using \eqref{eq:fH-1:2} that $u^R_\e\in H^1(I)$ 
with bound
\[
\begin{array}{rcl}
\| u_\e^R \|_{1,\e,I} 
&\simeq & \e |u^R_\e|_1 + \| u^R_\e \|_{L^2(I)} 
\leq C\|f\|_{-1,\e,I} 
\leq C\e^{-1} \|g_1(f)\|_{L^2(I)} 
\\
&=  &  C\e^{-1} \e^2 \| u_0' \|_{L^2(I)} 
\leq C \e \| u_0 \|_{H^1(I)} 
\leq C \e (\| f \|_{H^1(I)} + \| b \|_{H^1(I)})
\end{array} 
\]
where we used \eqref{u0} in the final bound. This proves \eqref{eq:u0H2} for $k=1$.

If $k=2$, we use $f = \e^2 u_0'' \in L^2(I)$ and conclude with \eqref{eq:apriori0}
the bound \eqref{eq:u0H2}.

\noindent 
Step 4 [Proof of \eqref{eq:uBLw}]:
by symmetry it suffices to prove \eqref{eq:uBLw}
for $u^{BL}_{\e,+}$.
We write $w_\e$ in place of $w^+_\e$ and define for $x\in I$ 
\begin{equation}\label{uBLbr}
\breve{u}_{+}^{BL}(x) 
:= -\frac{1}{2}u_0(1)e^{-(1-x)/(\beta_*\e)} (1+x)
 = -u_0(1)\frac{1+x}{2}(w_{\e}(x))^{-1} 
\;.
\end{equation}
For $0<\e \leq 1$ we have 
$\breve{u}_{+}^{BL} \in C^\infty(\overline{I})$,
as well as 
$\breve{u}_{+}^{BL}(1) = -u_0(1) = u^{BL}_{\e,+}(1)$, 
so that
\begin{equation}\label{uBLH}
U_{+}^{BL}(x) := u^{BL}_{+}(x) - \breve{u}_{+}^{BL}(x) \in H^k(I), \; U_{+}^{BL}(1) = 0\;.
\end{equation}
The function $\breve{u}_{+}^{BL}(x)$ satisfies the weighted 
$L^2_{w_\e}$-norm bound \eqref{eq:uBLw}.
It remains to establish the estimate \eqref{eq:uBLw} for $U_{+}^{BL}(x)$.

Therefore, by \eqref{BLs}, \eqref{uBLH},
\begin{equation} \label{eq:LeAux}
\begin{array}{rcl}
\mathcal{L}_{\e} U_{+}^{BL}&=& - \mathcal{L}_{\e} \breve{u}_{+}^{BL} 
\\
&=& -\frac{1}{2} u_0(1) e^{-(1-x)/(\beta_* \e)} \left\{ \frac{2\e}{\beta_*}+\frac{1+x}{\beta_*^2} -b(x)(1+x) \right\} 
     =: f^0_+(x;\e).
\end{array}
\end{equation}
Inspecting the expression for $f^0_+(x;\e)$, 
in either case $k=1,2$ holds $f^0_+(\cdot;\e)\in H^{k}(I)$ for every $0<\e\leq 1$.
Also, the definition \eqref{eq:weight} of the weight function $w_\e$
implies that 
$\Vert f^0_+(\cdot ;\e) \Vert_{L^2_{w_{\e}}(I)}$ 
is bounded independent of $\e\in (0,1]$.

\noindent
Step 4.1: 
We claim there exists a positive constant $C$ independent of $\e\in (0,1]$, 
such that
$$
\e^2 \left \Vert \left( U_{+}^{BL} \right)'' \right \Vert_{L^2_{w_{\e}}(I)} 
\leq 
C \Vert f^0_+(\cdot ;\e) \Vert_{L^2_{w_{\e}}(I)}.
$$
To prove this, we consider the bilinear form 
$B_\e(\cdot,\cdot): H^1(I)\times H^1(I)\to \R$ in \eqref{eq:BVPweak}
(i.e. \emph{without} homogeneous Dirichlet boundary conditions), 
and equip the space $H^1(I)$ with an
exponentially weighted energy norm $\| \circ \|_{1,w_\e}$ 
given by
\begin{equation}
\label{eq:energy_w}
\| v \|^2_{1,w_\e} := \int_I ( \e^2 (v')^2 + v^2) w_\e^2 dx \;.
\end{equation}

\noindent Step 4.2 [inf-sup stability of $B_\e(\cdot,\cdot)$ in exponentially weighted spaces]:
\newline
We denote $H^1(I)$ normed with the weighted norm by $H^1_{w_\e}(I)$. 
In $H^1(I)$ with this norm, the bilinear form $B_\e(\cdot,\cdot)$ admits 
the following inf-sup condition:
\begin{equation}
\label{infsup}
\inf_{ u \in H^1_{w_{\e}}(I)} \sup_{v \in H^1_{1/w_{\e}}(I)} 
\frac{B_{\e}(u,v)} { \Vert u \Vert_{1,{w_{\e}}} \Vert v \Vert_{1,{1/w_{\e}}}} 
\ge 
\gamma 
\end{equation}
for some positive constant $\gamma$, independent of $\e \in (0,1]$. 

To show \eqref{infsup}, let $0 \ne u \in H^1_{w_{\e}}(I)$ be given, 
and set $v_u = u w^2_{\e} \in {H^1_{1/w_{\e}}(I)}$. 
Then
\begin{eqnarray*}
\Vert v_u \Vert^2_{1,1/w_{\e}} &=& \int_I \left[\e^2 (v'_u)^2 + v_u^2 \right] w_{\e}^{-2} =\int_I \left[ \e^2 \left( u' w^2_{\e} + 2 u w'_{\e} w_{\e} \right)^2 +  u^2 w^4_{\e}\right] w_{\e}^{-2} \\
& = & \int_I \left[ \e^2 \left( u' w^2_{\e} - 2 u \frac{1}{\beta_{*} \e} w_{\e} w_{\e} \right)^2 + u^2 w^4_{\e}\right] w_{\e}^{-2} \\
& \leq &  2 \int_I \left[ \e^2 (u')^2 w^2_{\e} + \frac{4}{\beta^2_*}u^2 w_{\e}^2 +  u^2 w^2_{\e} \right] = 2 \int_I \left[ \e^2 (u')^2  + \left(\frac{4}{\beta^2_{*} }+ 1\right) u^2  \right] w^2_{\e} \\
& \leq & C_0 \Vert u \Vert^2_{1,w_{\e}},
\end{eqnarray*}
with 
$C_0 = \max\left\{2,  \left(\frac{4}{\beta^2_{*} }+ 1 \right)\right\}$
independent of $\e$.

The inf-sup condition \eqref{infsup} 
is established with $\gamma = C_1/C_0$ if we can show there exists
$C_1>0$ such that, for every $0<\e \leq 1$,
\begin{equation}
\label{coercivity}
B_\e(u,v_u) \ge C_1 \Vert u \Vert^2_{1,w_{\e}} \; .
\end{equation}
We have using that $w_\e' = -w_\e/(\e \beta_*)$
\begin{eqnarray}
B_\e(u,v_u) &=& \e^2 \int_I u'( u w^2_{\e})' + \int_I b u^2 w_{\e}^2  \notag \\
& = & \e^2 \int_I (u')^2 w^2_{\e}+ \int_I b u^2 w_{\e}^2 - 2 \frac{\e}{\beta_{*}} \int_I u'u w^2_{\e} \label{A2}.
\end{eqnarray}
The last term in (\ref{A2}) has no definite sign, 
so we proceed as follows: 
using $2 A B \leq \frac{2}{3} A^2 + \frac{3}{2} B^2$, 
we find
$$
\left \vert 2 \frac{\e}{\beta_{*}} \int_I u'u w^2_{\e}  \right \vert 
\leq 
\frac{ 2 \e^2}{3}  \int_I (u')^2 w^2_{\e} 
+ 
 \frac{3}{2 \beta_*^2} \int_I u^2 w^2_{\e} .
$$
Since
$\beta_*^2  = 4/ \underline{b}$, 
$$
\frac{3}{2 \beta_*^2} \int_I u^2 w^2_{\e} 
= 
\frac{3}{8}\underline{b} \int_I u^2 w^2_{\e} 
\leq 
\frac{3}{8} \int_I b u^2 w^2_{\e}
\;,
$$
and
the last term in \eqref{A2} can be absorbed into the first two, 
and (\ref{coercivity}) is shown.
As a result, (\ref{infsup}) holds with 
$C_1 =  \min\{1/3, 5\underline{b}/8 \}$.

\noindent Step 4.3:
The symmetry of $B_\e(\cdot,\cdot)$ implies also the complementary 
inf-sup condition holds with same constant $\gamma>0$ independent of $\e$.

\noindent
Step 4.4 [a-priori bound on $U_{+}^{BL}$ in exponentially weighted norm]: 
\newline
Using that $U_{+}^{BL}$ solves \eqref{eq:LeAux}, i.e.
\[
\mathcal{L}_{\e} U_{+}^{BL}= f^0_+(x;\e)
\quad\mbox{in}\quad I, \qquad U_{+}^{BL}(\pm 1)=0 \;,
\]
the weighted a-priori estimate implied by (\ref{infsup}) 
gives that there exists $C>0$ such that for $0<\e \leq 1$
\begin{equation}\label{eq:wapriori}
\Vert U^{BL}_{+} \Vert_{L^2_{w_{\e}}(I)} 
\leq 
\Vert U^{BL}_{+} \Vert_{1,{w_{\e}}}
\leq
C \Vert f_{+}^0 \Vert_{L^2_{w_{\e}}(I)} .
\end{equation}
We also have from \eqref{eq:LeAux}
$$
-\e^2 \left( U_{+}^{BL} \right)'' = f^0_+ - b U_{+}^{BL} \;.
$$
Taking the $\| \circ \|_{L^2_{w_{\e}}(I)}$-norm of both sides,
we get
\begin{equation}
\label{claim0}
\e^2 \left \Vert  \left( U_{+}^{BL} \right)''\right \Vert_{L^2_{w_{\e}}(I)} 
\leq
\Vert f^0_+ \Vert_{L^2_{w_{\e}}(I)} +\Vert b \Vert_{L^\infty(I)} \Vert U_{+}^{BL} \Vert_{L^2_{w_{\e}}(I)}.
\end{equation}
Combining \eqref{eq:wapriori} with (\ref{claim0}) gives
$$
\e^2 \left \Vert  \left( U_{+}^{BL} \right)''\right \Vert_{L^2_{w_{\e}}(I)} 
\leq
(1+C \Vert b  \Vert_{L^\infty(I)})\Vert f^0_+ \Vert_{L^2_{w_{\e}}(I)}.
$$

\noindent
Step 5 [conclusion]:
Using \eqref{uBLH}, and the definition \eqref{uBLbr},
$$
\e^2
\Vert (u^{BL}_+)'' \Vert_{L^2_{w_{\e}}(I)}
\leq 
\e^2
\left(
 \Vert \left( U_{+}^{BL} \right)'' \Vert_{L^2_{w_{\e}}(I)} 
+\Vert \left( \breve{u}_{+}^{BL} \right)'' \Vert_{L^2_{w_{\e}}(I)}
\right)
$$
from which the desired result \eqref{eq:uBLw} follows.
\end{proof}
The preceding result, Prop.~\ref{prop:Reg1}, covers in particular the (classical)
case when $b(x) = const. = \underline{b} = \bar{b} > 0$.
For subsequent reference, we state it.
\begin{corollary}\label{cor:bconst} [Constant reaction term $b(x)$]
Under the assumptions of Prop.~\ref{prop:Reg1}, with additionally 
constant reaction coefficient $b(x) = b > 0$, the statement 
of Prop.~\ref{prop:Reg1} holds, up to constants to accommodate
the boundary conditions in \eqref{BLs}, 
with 
\begin{equation}\label{eq:uBLexp}
	u^{BL}_{\e,\pm}(x) = C^{\pm} \exp\left( - \sqrt{b}(1 \mp x)/\e \right) + const. \;,
\end{equation}
where the constants $C^{\pm}$ are bounded from above independently of $\e \in (0,1]$. 
For example, for a right-hand side $f$ with $f(-1)=-1$, 
$u^{BL}_{\e,-}$ 
is determined from
\begin{eqnarray*}
-\e^2 (u^{BL}_{\e,-})''(x) + b u^{BL}_{\e,-}(x) &=& 0 \; , \; x \in I=(-1,1), \\
u^{BL}_{\e,-}(-1) = 1, u^{BL}_{\e,-}(1) &=& 0,
\end{eqnarray*}
as
$$
u^{BL}_{\e,-} (x) = \frac{e^{-\sqrt{b}(1+x)/\e} - e^{-\sqrt{b}(1-x)/\e}}{\left( e^{\sqrt{b}/\e} - e^{-\sqrt{b}/\e}\right) \left( e^{\sqrt{b}/\e} + e^{-\sqrt{b}/\e}\right)}
\in 
{\rm span}\{ \exp(- \sqrt{b}(1\pm x)/\e) \} 
\;.
$$
\end{corollary}
\begin{remark}
\label{rmk:PtswiseBd}
		Pointwise bounds for the boundary layer terms
		$u_{\e,\pm}^{BL}$, as typically arising in boundary layer
		decompositions, remain valid under the regularity assumption $b,f \in H^1(I)$. 
                See, for example, \cite[Lemma~6.2]{ORiordanShishkin2012}. 
                In particular, there exists a constant $C(b,f)>0$ 
                independent of $\e \in (0,1]$ such that
		\begin{equation}\label{eq:uBLpwBound}
			|u^{BL}_{\e,\pm} (x)| \le C 
			%\exp(- \sqrt{\underline{b}} (x \mp 1)/ \e ), \quad x \in [-1,1].
			\exp \left(-\frac{1\mp x}{\e \beta_*}\right), \quad x \in [-1,1].
		\end{equation}
		
		We provide a proof to keep the exposition self-contained:
		By symmetry, we only consider the boundary layer at the left endpoint $x=-1$.
		Since $u^{BL}_{\e,-}$ solves the homogeneous problem \eqref{BLs} and $b \in H^1(I) \hookrightarrow C^{0,1/2}(\overline{I})$, (classical) elliptic regularity theory implies $u^{BL}_{\e,-} \in C^{0}(\overline{I}) \cap C^{2,1/2}(I)$.
		
		Define
		$$
			\phi(x) := 
	%		\exp(- \sqrt{\underline{b}} (x+1)/ \e )
			\exp \left(-\frac{1 +  x}{\e \beta_*}\right)
			, \qquad x \in [-1,1].
		$$
		Then, for any $0 < \e \le 1$,
		$$ \mathcal{L}_{\e} \phi = \left(b(x) - \frac{\underline{b}}{4}\right) \phi \ge 0.$$
		Moreover,
		$$|u^{BL}_{\e,-} (-1)| = |-u_0(-1)| = |-u_0(-1)| \phi(-1), \qquad 
                  |u^{BL}_{\e,-} (1)| = 0 \le |-u_0(-1)| \phi(1).
                $$
		Introducing
		$$ \psi_{\pm}(x) =  |u_0(-1)| \phi(x) \pm  u^{BL}_{\e,-}(x),$$
		we obtain
		$$\psi_{\pm}(-1) \ge 0, \qquad \psi_{\pm}(1) \ge 0.$$
		By the strong maximum principle,
		\begin{equation*}
			|u^{BL}_{\e,-} (x)| \le |u_0(-1)| 
%			\exp(- \sqrt{\underline{b}} (x+1)/ \e )
			\exp \left(-\frac{1 + x}{\e \beta_*}\right)
			, \quad x \in [-1,1].
		\end{equation*}
	\end{remark}
%
%%%%%%%%%%%%%%%%%%%%%%%%%%%%%%%%%%%%%%%%%%%%%%%%%%%%%%%%%%%
\section{Robust ${\mathbb P}_1$ Finite-Element Approximation}
\label{sec:approx}
%%%%%%%%%%%%%%%%%%%%%%%%%%%%%%%%%%%%%%%%%%%%%%%%%%%%%%%%%%%
Robust, i.e. $\e$-explicit, numerical approximation of 
the solution of (\ref{eq:de})--(\ref{eq:bc}) 
by the Finite Volume Method (FVM) and by the 
Finite Element Method (FEM) was studied in many works.
We mention only \cite{RST2nd} and the references there.
These references address the setting where the data $f$
and $b$ in \eqref{eq:de}, \eqref{eq:bc} have finite Sobolev regularity,
as opposed to \cite{SS1996,Melenk1997,OSX24_1085} where
in \eqref{eq:de} reaction coefficients $b$ and source terms $f$ 
analytic in $\overline{I}$ were considered.
%%%%%%%%%%%%%%%%%%%%%%%%%%%%%%%%%%%%%%%%%%%%%%%%%%%%%%%%%%
\subsection{Galerkin Approximation}
\label{sec:GalApr}
%%%%%%%%%%%%%%%%%%%%%%%%%%%%%%%%%%%%%%%%%%%%%%%%%%%%%%%%%%
Galerkin approximations of \eqref{eq:de} -- \eqref{eq:bc}
will be based on the weak formulation \eqref{eq:BVPweak}.
In order to define the discrete version of (\ref{eq:BVPweak}), 
for $N\in\N$,
let $\Delta = \{ x_j \}_{j=0}^{N}$ be a
partition of $I$ with $-1 = x_0 < x_1 < ... < x_N = 1$
and set 
$I_j = [x_{j-1} , x_{j} ]$, $h_j = x_{j} - x_{j-1}, j=1, \ldots, N, h = \max_{1\leq j \leq N} h_j$. 
With $\mathcal{P}_1(I)$ the space of polynomials of degree at most $1$ on $I$, 
we define 
\begin{eqnarray}
\label{eq:Sp}
\mathcal{S}^1(\Delta) &:=& \{ w \in H^1(I) : w|_{I_j} \in \mathcal{P}_1(I_j), j=1,\dots,N \} ,
\\
\mathcal{S}^1_0(\Delta) &:=& \mathcal{S}^1(\Delta) \cap H^1_0(I) \label{eq:Sp0}.
\end{eqnarray}
The Galerkin discretization of (\ref{eq:BVPweak}) 
then reads: 
find $u^\Delta_{\e} \in \mathcal{S}^1_0(\Delta)$ 
such that
for all $v\in \mathcal{S}^1_0(\Delta)$, 
there holds 
\begin{equation}
\label{eq:discrete}
B_\e(u^\Delta_{\e},v)
=
\int_{-1}^1 \left\{ \e^2 \left(u^\Delta_{\e}\right)' v' + b u^\Delta_{\e} v \right\} dx = \int_{-1}^1 f v dx.
\end{equation}
Associated with the above problem, 
in the \emph{energy norm} given by (\ref{eq:energy}),
the best approximation property holds:
\begin{equation}
\label{eq:best_approx}
 \forall \; v \in \mathcal{S}^1_0(\Delta):
\quad 
\Vert u_{\e} - u^{\Delta}_{\e} \Vert_{1,\e,I} 
\leq 
\Vert u_{\e} -v \Vert_{1,\e,I} 
\;.
\end{equation}
The FE error is bounded by the best approximation error, in the energy norm \eqref{eq:energy},
uniformly for $0\leq \e \leq 1$.
%%%%%%%%%%%%%%%%%%%%%%%%%%%%%%%%%%%%%%%%%%%%%%%%%%%%%%%%%%
\subsection{Layer-adapted Meshes}
\label{sec:LayAdMes}
%%%%%%%%%%%%%%%%%%%%%%%%%%%%%%%%%%%%%%%%%%%%%%%%%%%%%%%%%%
Layer adapted meshes are well known to allow for $\e$-uniform 
approximation of exponential boundary layer functions \eqref{eq:uBLexp} 
and, hence, also of $u_\e$. 
See, e.g., \cite{Linss1985,RST2nd} and references there.
We briefly recap the construction of layer-adapted meshes $\Delta$ in $I$.

We begin with splitting the domain $\overline{I}=[-1,1]$
%into layer and coarse regions 
as follows:
\begin{equation}\label{eq:split}
I_{\text {layer }}^{-}:=[-1,-1+\lambda] 
\; , \; 
I_{\text {coarse }}:=[-1+\lambda, 1-\lambda] 
\; , \; 
I_{\text {layer }}^{+}:=[1-\lambda, 1] \text {, }
\end{equation}
where $\lambda(\e, N)$ 
is chosen as
\begin{equation}\label{eq:trpt}
\lambda = \min\{ \theta \e \ln N, 1/4\} \leq \frac{1}{4},
\end{equation}
where $\theta > 0$ is a (fixed) constant 
that is independent of $\e$ and of $N$ which is 
to be specified shortly (see, e.g., \eqref{eq:theta} below)
and $N+1$ is the total number of nodal points. 
The points 
$-1 + \lambda$ and $1 - \lambda$ 
are referred to as the \emph{transition points} of the mesh.

If $\lambda \geq 1/4$, the problem is not singularly perturbed, and
standard approximation bounds on (quasi-) uniform meshes will imply
the asserted bounds. 
%{\color{magenta} How could $\lambda$ be greater than $1/4$? Its definition prohibits that. For $\lambda=1/4$, we
%find ourselves in the $N \ge \e^{-1}$ regime, for which Martin Stynes told me that the somewhat bold claim made
%has not been proved anywhere in the literature, and he himself does not believe it. I propose we just make the assumption
%(\ref{eq:eN}) in order to study the singularly perturbed case only.}
We thus assume 
\begin{equation}
\label{eq:eN}
\lambda < 1/4 \;\; \mbox{as well as}\;\; 1 \leq N < \e^{-1}.
\end{equation}
%
%when $N>1/\e$, the boundary layers are resolved.
The $N+1$ nodal points\footnote{We assume, without loss of generality, 
that $N$ is divisible by 4.} 
are given by
$$
	x_{i}=\left\{ 
	\begin{array}{cc}
		-1+\theta \varepsilon \phi \left( \frac{2i}{N}\right) , & i=0,\ldots ,N/4 \\ 
		\frac{4i}{N}(1-\lambda )-2+2\lambda , & i=N/4,\ldots ,3N/4 \\ 
		1-\theta \varepsilon \phi \left( 2-\frac{2i}{N}\right)  & i=3N/4,\ldots ,N%
	\end{array}%
	\right. 
$$
where $\phi$ is a so-called \emph{mesh-generating function}
with the following properties:
\begin{enumerate}
\item $\phi$ is monotonically increasing.
\item $\phi(0)=0, \phi(1/2) = \ln (\alpha N) \text{ for some }\alpha \in \mathbb{R}^+$.
\item $\phi$ is piecewise differentiable with $\max \phi' \leq C N$.
\end{enumerate}
The mesh size $h_i := |I_i| = x_i - x_{i-1}$ satisfies (see, e.g., \cite{Linss1985})
\begin{equation}\label{eq:hj_exp}
h_{i}
\leq 
CN^{-1}
\left\{ 
\begin{array}{cc}
 \e \max_{I_i}| \phi' | , & i=0,\ldots ,N/4, \\ 
 & i=N/4,\ldots ,3N/4, \\ 
 \e  \max_{I_i}| \phi' |,  & i=3N/4,\ldots ,N.%
\end{array}%
\right. 
\end{equation}
Related to $\phi$ is the so-called \emph{mesh characterization function}
\begin{equation}\label{eq:meshchar}
\psi = e^{-\phi},
\end{equation}
and there holds $\phi = - \ln \psi, \phi' = - \psi' (\psi)^{-1}$.
Meshes generated as above, 
are called (generalized if $\alpha \ne 1$) S-type meshes 
(see, e.g., \cite{FranzCAX2018} and the references therein). 
S-type meshes comprise and unify several classical
layer-adapted mesh constructions.
%%%%%%%%%%%%%%%%%%%%%%%%%%%%%%%%%%%%%%%%%%%%%%%%%%%%%%%%
\subsubsection{Shishkin mesh $\Delta_S$}
\label{sec:Shish}
For the Shishkin mesh \cite{Shishkin1989}, we use $\phi$ in \eqref{eq:meshchar} with $\alpha =1$,
\begin{equation}\label{eq:mshchShisk}
	\phi(t) = 2t \ln N \; , \; \psi(t) = N^{-2t}.
\end{equation}
There holds $\max|\psi'| = 2 \ln N$.
The resulting mesh, denoted by $\Delta_S$,
is uniform on the sub intervals $I^\pm_{\text{layer}}$, $I_{\text{coarse}}$ in \eqref{eq:split}.
This patchwise uniform structure has made $\Delta_S$ a popular choice
(see \cite{Linss1985,RST2nd,Shishkin1989} and the references there).

In the following we will assume
\begin{equation*}
	\lambda  = \theta \e \ln(N), \quad \text{and} \quad N \le \e^{-1}.
\end{equation*}
At this point we also specify $\theta$ as
\begin{equation}
	\label{eq:theta}
	\theta = 2 \beta_*,
\end{equation}
and introduce
\begin{equation*}
	\lambda_{\e} := \theta \e \ln(1/\e) > \lambda.
\end{equation*}
With these choices, the interval lengths $h_i = |I_i| $ satisfy
\begin{equation}\label{eq:mesh-lengths-fine}
	h_i = 4 \theta \e N^{-1} \ln(N),
	\quad \text{for } i  = 1,\dots,N/4,\,3N/4 +1,\dots,N,
\end{equation}
and
\begin{equation}\label{eq:mesh-lengths-coarse}
	h_i = 3 N^{-1} \leq h_i \leq 4 N^{-1} , \quad \text{for } i  = N/4+1,\dots,3N/4.
\end{equation}
\begin{proposition}
\label{prop:shish}
For $0<\e\leq 1$ and $b,f \in H^1(I)$,
let $u_{\e}$ be the solution to \eqref{eq:de}--\eqref{eq:bc}.
Let $u_{\e}^{\Delta_S}\in \mathcal{S}^1_0(\Delta_S)$ be its  
CpwL $\mathbb{P}_1$-FE approximation on the Shishkin mesh $\Delta_S$ 
at length scale $\e\in (0,1]$ with $1\leq N \leq 1/\e$ many  intervals.

Then there exists a positive constant $C$, independent of $\e$, 
such that 
for every $0<\e \leq 1$ and 
for every $1\leq N \leq 1/\e$ holds
\begin{equation}\label{eq:ShiskBd}
\Vert u_{\e} - u_{\e}^{\Delta_S} \Vert_{1,\e,I} \leq C N^{-1} \ln N.
\end{equation}
\end{proposition}
\begin{proof}
	We show \eqref{eq:ShiskBd} by constructing a CpwL interpolant of $u_\e$ which satisfies the bound \eqref{eq:ShiskBd}. 
	By $\e$-uniform quasioptimality of the Finite Element projection in the energy norm,
	\eqref{eq:ShiskBd} will follow.
	
	Proposition \ref{prop:Reg1} gives the decomposition \eqref{eq:decomp} for $u_{\e}$.
	With \eqref{eq:decomp},
	we write the approximation, with obvious notation, as
	$$
	u_{\e}^{\Delta_S} = u_0^{\Delta_S} + u^{BL, \Delta_S}_{\e,-} + u^{BL,\Delta_S}_{\e,+}
	+ u_{\e}^{R,\Delta_S}
	\;.
	$$
	Using the triangle inequality
	we bound each error contribution, for $N \in [1, 1/\e]$.
	
	\medskip
	\noindent\textbf{Remainder $u_\e^R$.}
	By \eqref{eq:u0H2} for $k=1$,
	in the range $\e \leq 1/N \leq 1$
	the remainder satisfies
	$\Vert u_{\e}^R \Vert_{1,\e,I} \leq C \e \leq C N^{-1}$. 
	Note that $ u_{\e}^R(\pm 1) = 0$ (see \eqref{rem}).
	We therefore may choose $u_{\e}^{R,\Delta_S} = 0$.
	
	\medskip
	\noindent\textbf{Smooth part $u_0$.}
	For $k=1$, the smooth part
	$u_0 \in H^1(I)$ is given by \eqref{u0}.
	We have (cf. \eqref{eq:u0H2})
	with a continuous, piecewise linear
	interpolant $u_0^{\Delta_S}$ that
	\begin{equation}\label{eq:rate_smooth}
		\Vert u_0 - u_0^{\Delta_S} \Vert_{L^2(I)} + \e \Vert (u_0 - u_0^{\Delta_S})' \Vert_{L^2(I)}
		\leq C (N^{-1} +\e)
		\leq C N^{-1},
	\end{equation}
	with $C>0$ a constant independent of $\e$.
	Here we used that $0 < \e \leq 1/N \leq 1$.
	Additionally,
	$(u_0 - u_0^{\Delta_S})(\pm1) = 0$.
	
	\medskip
	\noindent\textbf{Boundary layer terms $u^{BL}_{\e,\pm}$.}
	By symmetry, we will only consider the layer at the left endpoint $x=-1$.
	We adapt the approach taken in \cite{SunStynes95} to our low regularity setting. 
	
	In the following we denote by $\mathcal{I} (u^{BL}_{\e,-})$ the CpwL nodal interpolant of $u^{BL}_{\e,-}$ on the Shishkin mesh. 
	We set
	\begin{equation*}
		i_0 := N/4, \qquad i_{1} := \max \{ i : I_{i} \cap (-1 + \lambda, -1 + \lambda_{\e}) \neq \emptyset \},
	\end{equation*}
	and
	\begin{equation*}
		K_1 := [x_0,x_{i_0}], \qquad K_2 := [x_{i_0}, x_{i_1}], \qquad  K_3 := [x_{i_1}, x_{N}].
	\end{equation*}
	Then, $\overline{I} = K_1 \cup K_2 \cup K_3$.
	We estimate the error contributions from each of these sub-regions:
	
	On $K_1$, 
        using classical bounds for the interpolation error
        (e.g. \cite[Prop.~1.5]{ErnGuermond2004}) 
        and the bound on the interval lengths \eqref{eq:mesh-lengths-fine}, 
        we have for $\ell=0,1$,
	\begin{align*}
		\| \left(u_{\e, -}^{BL} - \mathcal{I}(u_{\e, -}^{BL})\right)^{(\ell)} \|_{L^2(K_1)}^2
		&= \sum_{i=1}^{i_{0}} \|\left(u_{\e, -}^{BL} - \mathcal{I}(u_{\e, -}^{BL})\right)^{(\ell)}\|_{L^2(I_{i})}^2\\
		&\leq C \sum_{i=1}^{i_{0}} h_{i}^{2(2-\ell)} \|\left(u_{\e, -}^{BL}\right)''\|_{L^2(I_{i})}^2\\
		&\leq C \, (\e N^{-1} \ln(N))^{2(2-\ell)} \|\left(u_{\e, -}^{BL}\right)''\|_{L^2([-1, -1 + \lambda])}^2.
	\end{align*}
	
	Moreover, by the a-priori estimate \eqref{eq:uBLw} and the fact that $w_{\e}^{-} \ge 1$ on $\overline{I}$, we have $\| \left(u_{\e, -}^{BL}\right)'' \|_{L^2([-1, -1 + \lambda])}^2 \leq \| \left(u_{\e, -}^{BL}\right)'' \|_{L_{w_{\e}}^2([-1, -1 + \lambda])}^2 \leq C \e^{-4}$. Hence, we obtain that
	\begin{equation*}
		\| \left(u_{\e, -}^{BL} - \mathcal{I}(u_{\e, -}^{BL})\right)^{(\ell)} \|_{L^2(K_1)}^2 \leq C \e^{-4} (\e N^{-1} \ln(N))^{2(2-\ell)} \leq C \e^{-2\ell} (N^{-1} \ln(N))^{2(2-\ell)}.
	\end{equation*}
	Combining the two estimates yields
	\begin{equation}
		\label{eq:interp-err-bound-K1}
		\| u^{BL}_{\e,-} - \mathcal{I}(u^{BL}_{\e,-}) \|_{1,\e,K_1}
		%\le C \bigl((N^{-1} \ln(N))^{2} + (N^{-1} \ln(N))^{4} \bigr)^{1/2} 
		\le C N^{-1} \ln(N).
	\end{equation}
	
	On $K_2$, we instead use the pointwise bound \eqref{eq:uBLpwBound} to control the interpolation error in the $L^2$-norm, while the $H^1$-stability of $\mathcal{I}$ (see, e.g., \cite[Prop.~1.4]{ErnGuermond2004}) yields a bound for the interpolation error in the $H^1$-seminorm.
	Indeed,
	\begin{align*}
		\| \left(u^{BL}_{\e,-} - \mathcal{I}(u^{BL}_{\e,-})\right)' \|_{L^2(K_2)}^2 &= \sum_{i=i_0}^{i_1} \| \left(u^{BL}_{\e,-} - \mathcal{I}(u^{BL}_{\e,-})\right)' \|_{L^2(I_{i})}^2\\
		&\leq C \sum_{i=i_0}^{i_1} \| \left(u^{BL}_{\e,-}\right)' \|_{L^2(I_{i})}^2 \le C \| \left(u^{BL}_{\e,-}\right)' \|_{L^2([x_{i_0}, x_{i_1}])}^2.
	\end{align*}
	Then, using that $x \mapsto (w_{\e}^{-}(x))^{-2}$ is monotonically decreasing on $\mathbb{R}$ and the a priori estimate \eqref{eq:wapriori}, we estimate
	\begin{align*}
		\| \left(u^{BL}_{\e,-}\right)' \|_{L^2([x_{i_0}, x_{i_1}])}^2 &= \int_{x_{i_0}}^{x_{i_1}} \left(\left(u^{BL}_{\e,-}(x)\right)'\right)^2 (w_{\e}^{-}(x))^2 (w_{\e}^{-}(x))^{-2} dx\\
		&\le \max_{x \in [x_{i_0}, x_{i_1}]} (w_{\e}^{-}(x))^{-2}  \| \left(u^{BL}_{\e,-}\right)' \|_{L_{w_{\e}}^2([x_{i_0}, x_{i_1}])}^2\\
		&\le C \e^{-2} (w_{\e}^{-}(x_{i_0}))^{-2}.
	\end{align*}
	From the definition of the transition point 
        $x_{i_0} = -1 + \lambda$ (cf. \eqref{eq:trpt} with \eqref{eq:theta}), 
        it follows that
	\begin{equation*}
	(w_{\e}^{-}(x_{i_0}))^{-2} 
         = \exp\left(- 2 \frac{-1 + x_{i_0}}{\e \beta_*}\right) 
         = \exp\left(- 2 \frac{\theta \e \ln N}{\e \beta_*}\right) 
         = N^{-4}.
	\end{equation*}
	For the other term in the $1,\e$-norm we estimate,
	\begin{align*}
	\| u^{BL}_{\e,-} - \mathcal{I}(u^{BL}_{\e,-}) \|_{L^2(K_2)}^2 
       &= \sum_{i=i_0}^{i_1} \| u^{BL}_{\e,-} - \mathcal{I}(u^{BL}_{\e,-}) \|_{L^2(I_{i})}^2 
       \le \sum_{i=i_0}^{i_1} | I_i | \| u^{BL}_{\e,-} - \mathcal{I}(u^{BL}_{\e,-}) \|_{L^\infty(I_{i})}^2
       \\
       &\le C \sum_{i=i_0}^{i_1} | I_i | \| u^{BL}_{\e,-} \|_{L^\infty(I_{i})}^2
       \\
       &\le C\, {|K_2|}\| u^{BL}_{\e,-} \|_{L^{\infty}([x_{i_0}, x_{i_1}])}^2.
       \end{align*}
       Again, using the definition of the transition point 
       together with the pointwise bound \eqref{eq:uBLpwBound}, 
       we obtain
	\[
	\| u^{BL}_{\e,-} \|_{L^{\infty}([x_{i_0}, x_{i_1}])}^2 
        \le \frac{C}{(w_{\e}^{-}(x_{i_0}))^{2}} \le C N^{-4}.
	\]
	Combining the two estimates yields with $|K_2| \le 1$,
	\begin{equation}
	\label{eq:interp-err-bound-K2}
	\| u^{BL}_{\e,-} - \mathcal{I}(u^{BL}_{\e,-}) \|_{1,\e,K_2}^2 
        \le C (N^{-4} + |K_2| N^{-4}) \le C N^{-4}.
	\end{equation}
	
	On $K_3$, we proceed as for $K_1$, 
        but use the bound on the coarse interval lengths \eqref{eq:mesh-lengths-coarse}. 
        Hence, for $\ell=0,1$,
	\begin{align*}
	\| \left(u_{\e, -}^{BL} - \mathcal{I}(u_{\e, -}^{BL})\right)^{(\ell)} \|_{L^2(K_3)}^2
	&= \sum_{i=i_1}^{N} \| \left(u_{\e, -}^{BL} - \mathcal{I}(u_{\e, -}^{BL})\right)^{(\ell)} \|_{L^2(I_{i})}^2\\
	&\leq C \sum_{i=i_1}^{N} h_{i}^{2(2-\ell)} \| \left(u_{\e, -}^{BL}\right)'' \|_{L^2(I_{i})}^2\\
	&\leq C  N^{-2(2-\ell)} \sum_{i=i_1}^{N} \| \left(u_{\e, -}^{BL}\right)'' \|_{L^2(I_{i})}^2\\
	&\leq  C  N^{-2(2-\ell)} \| \left(u_{\e, -}^{BL}\right)'' \|_{L^2([x_{i_1}, x_{N}])}^2.
	\end{align*}
	To bound $\| \left(u_{\e, -}^{BL}\right)'' \|_{L^2([x_{i_1}, x_{N}])}^2$,  
        we use that $[x_{i_1}, x_{N}] \subset [-1 + \lambda_{\e},1]$ and the a priori bound \eqref{eq:uBLw}. 
	Proceeding similarly as on $K_2$ yields
	\begin{align*}
		\| \left(u_{\e, -}^{BL}\right)'' \|_{L^2([-1 + \lambda_{\e},1])}^2 &\le \max_{x \in [-1 + \lambda_{\e},1]} (w_{\e}^{-}(x))^{-2}  \| \left(u^{BL}_{\e,-}\right)'' \|_{L_{w_{\e}}^2([-1+\lambda_{\e},1])}^2\\
		&\le C \e^{-4} (w_{\e}^{-}(-1 + \lambda_{\e}))^{-2} \le C,
	\end{align*}
	and so
	\begin{equation}
	\label{eq:interp-err-bound-K3}
	\| u_{\e, -}^{BL} - \mathcal{I}(u_{\e, -}^{BL}) \|_{1,\e,K_3}^2 \le C (\e^2 N^{-2} + N^{-4}).
	\end{equation}
	
	Combining the estimates 
        \eqref{eq:interp-err-bound-K1}, \eqref{eq:interp-err-bound-K2},  
        and \eqref{eq:interp-err-bound-K3} yields
	\begin{align*}
		\|u_{\e, -}^{BL} - \mathcal{I}(u_{\e, -}^{BL})\|_{1,\e,I} 
                \leq C (N^{-1} \ln(N) + \e N^{-1} + N^{-2}) \le C N^{-1} \ln(N).
	\end{align*}
	Letting $u^{BL, \Delta_S}_{\e,\pm} := \mathcal{I}(u^{BL}_{\e,\pm})$,
	we finally observe that the CpwL approximant
	\[
	u_0^{\Delta_S} + u^{BL,\Delta_S}_{\e,-} + u^{BL,\Delta_S}_{\e,+} \in  S^1(I,\Delta_S)
	\]
	satisfies the zero BCs at $\pm 1$, by construction, as well as the bound \eqref{eq:ShiskBd}.
	Combining the above, the proof is complete.
\end{proof}
\begin{remark}\label{rmk:Shk=2}
When $k=2$, i.e. for $b,f \in H^2(I)$, 
the bound (\ref{eq:rate_smooth}) 
holds for the approximation error in the remainder, 
$u_{\e}^{R} - u_{\e}^{R,\Delta_S}$, as well. 
Indeed, from (\ref{eq:u0H2}), 
the smooth part $u_0$ belongs to $H^2(I)$ and 
the remainder $u_\e^R$ satisfies
$$
\Vert u_{\e}^{R} \Vert_{L^2(I)}  \leq \Vert u_{\e}^{R} \Vert_{1, \e, I} \leq C \e^2.
$$
Also, from (\ref{rem}) we have 
$$
-\e^2 (u_\e^R)'' + b u_\e^R = \e^2 u_0'' \; , \; u_\e^R(\pm 1) = 0,
$$
thus
$$
\Vert (u_\e^R)'' \Vert_{L^2(I)} \leq \Vert u_0'' \Vert_{L^2(I)} +  \e^{-2}\Vert b u_\e^R \Vert_{L^2(I)} \leq C,
$$
since $b$ is bounded. 
This means $u_\e^R \in H^2(I)$, 
and by classical results the continuous, piecewise linear interpolation
approximation on a quasi-uniform mesh $u_\e^{R,\Delta}$, 
satisfies 
\begin{eqnarray*}
\Vert u_\e^R - u_\e^{R,\Delta} \Vert_{L^2(I)} &\leq& C N^{-2} \Vert (u_{\e}^{R})'' \Vert_{L^2(I)} , 
\\
\Vert (u_\e^R - u_\e^{R,\Delta})' \Vert_{L^2(I)} &\leq& C N^{-1} \Vert (u_{\e}^{R})'' \Vert_{L^2(I)}.
\end{eqnarray*}
Therefore, for $1\leq N \leq \e^{-1}$, 
\begin{eqnarray*}
\Vert u_\e^R - u_\e^{R,\Delta} \Vert_{1,\e,I} 
&\leq& \Vert u_\e^R - u_\e^{R,\Delta} \Vert_{L^2(I)} + \e \Vert (u_\e^R - u_\e^{R,\Delta})' \Vert_{L^2(I)} 
\\
&\leq& C ( N^{-2} + \e N^{-1} )  \leq C N^{-2} \;,
\end{eqnarray*}
which is the optimal $L^2(I)$-rate.

For \emph{any $N\geq 1$ and any} $\e\in (0,1] $
(comprising in particular $\e=1$), 
one has
\begin{eqnarray*}
\Vert u_\e^R - u_\e^{R,\Delta} \Vert_{1,\e,I}
&\leq& C N^{-1}
\end{eqnarray*}
which is the optimal $H^1(I)$ rate.
\end{remark}
%}
%%%%%%%%%%%%%%%%%%%%%%%%%%%%%%%%%%%%%%%%%%%%%%%%%%%%%%%%
\subsubsection{Exponential mesh $\Delta_{eXp}$}
%%%%%%%%%%%%%%%%%%%%%%%%%%%%%%%%%%%%%%%%%%%%%%%%%%%%%%%%
In \cite{XenoThesis} a class of  meshes $\Delta_{eXp}$ 
was derived for the robust, asymptotically optimal, approximation of the
explicit boundary layer function $e^{-\beta x/ \e}, \beta \in \mathbb{R}^+$, 
appearing in \eqref{eq:uBLexp}. 
The mesh points of $\Delta_{eXp}$ 
are ``exponentially graded'' towards the endpoints, in an $\e$-dependent way.
In \cite{FranzCAX2018} it was shown that 
the exponential mesh $\Delta_{eXp}$ may be 
regarded as a generalized S-type mesh with $\alpha=1/2$,
and
\begin{equation}\label{eq:phieXp}
\phi(t) = - \ln\left( 1 - 2t (1-2N^{-1}) \right) \; , \; \psi(t) = 1 - 2t(1-2N^{-1}),
\end{equation}
with  $\max|\psi' | \leq 2$.
For $\theta > 0$ fixed, the $N+1$ nodal points are given by
	\[
	x_i = 
	\begin{cases}
		-1 + \e \theta \phi\left(\frac{2i}{N}\right), & i=0,\dots,N/4\;,\\
		\frac{4i}{N}(1 - \theta \e \ln(N/2)) - 2 + 2 \theta \e \ln(N/2) & i=N/4,\dots,3N/4\;,\\
		1 - \e \theta \phi\left(2 - \frac{2i}{N}\right) & i=3N/4,\dots,N\;.
	\end{cases}
	\]
	Here, the mesh size $h_i = | I_{i}| =  x_{i} - x_{i-1}$ of $I_{i} \subset I_{{\text{layer}}}^{-}$, for $i=1,\dots,N/4$, satisfies
	\begin{align}
		\nonumber h_i = x_i - x_{i-1} &= \e \theta \left[\phi\left(\frac{2i}{N}\right) -  \phi\left(\frac{2(i-1)}{N}\right) \right]\\
		\nonumber&= \e \theta \left[- \ln\left( 1 - 2 \left(\frac{2i}{N}\right) (1-2N^{-1}) \right) + \ln\left( 1 - 2\left(\frac{2(i-1)}{N}\right) (1-2N^{-1}) \right)\right]\\
		\nonumber&= \e \theta \ln\left(\frac{1 - 2\left(\frac{2(i-1)}{N}\right) (1-2N^{-1})}{1 - 2 \left(\frac{2i}{N}\right) (1-2N^{-1})} \right)\\
		&\label{eq:hj_exp_psi}= \e \theta \ln\left(\frac{\psi\left(\frac{2(i-1)}{N}\right)}{\psi \left(\frac{2(i-1)}{N}\right) - \delta} \right)\;,
	\end{align}
where $\delta = \frac{4}{N} (1 - 2 N^{-1})$.
%{\color{magenta}Note that in the left layer region, i.e.
%$t \in [0,1/2-1/N]$, there holds
%$$
%x_j = -1 + 2\e \phi(j/N) \; , \; j=0, 1, \dots, m=\frac{N}{2}-1,
%$$
%and in particular
%%
%$$
%x_m = -1 + 2\e \phi(1/2 - 1/N).
%$$
%We have using a Taylor expansion
%$$
%\phi(1/2 - 1/N) = -\ln ( 1 - 4/N + 4/N^2) = \frac{4}{N} + O(1/N^2),
%$$
%and thus
%$$
%x_m \approx \frac{4\e}{N} + O(\e/N^2).
%$$
%The mesh size is the first interval is given by
%%
%$$
%h_1 = x_1 - x_0 = -1+2\e \phi(1/N) \approx \frac{2\e}{N},
%$$
%from which it follows that $h_i =O(\e/N),i=1,\ldots,m$, in the left layer region.
%The same holds for the right layer region.
%}

We next adapt the arguments of \cite{XenoThesis,FranzCAX2018} 
to the approximation of the 
boundary layers $u^{BL}_{\e,\pm}$ obtained in Prop.~\ref{prop:Reg1},
using the weighted $H^2$ estimates \eqref{eq:uBLw} 
rather than the explicit, pointwise representation \eqref{eq:uBLexp} 
(which are valid for constant reaction coefficient) 
or pointwise bounds for the second derivatives as e.g. \cite{FranzCAX2018}.
\begin{proposition}
\label{prop:eXp}
For $0<\e\leq 1$ and $b,f \in H^1(I)$,
let $u_{\e}$ be the solution to \eqref{eq:de}--\eqref{eq:bc}.
Let $u_{\e}^{\Delta_{eXp}}\in \mathcal{S}^1_0(\Delta_{eXp})$ be its  
CpwL FE approximation on the eXp mesh $\Delta_{eXp}$ 
at length scale $\e\in (0,1]$ with $1\leq N \leq 1/\e$ intervals.

Then there exists a positive constant $C$, independent of $\e$, 
such that for every $1\leq N \leq 1/\e$ holds
\begin{equation}\label{eq:eXpBd}
\Vert u_{\e} - u_{\e}^{\Delta_{eXp}} \Vert_{1,\e,I} \leq C N^{-1}.
\end{equation}
\end{proposition}
\begin{proof}
As in the proof of Proposition~\ref{prop:shish}, 
we show the statement by constructing a CpwL interpolant of $u_{\e}$ and 
use the quasioptimality of the Finite Element projection in the energy norm to conclude.
We use the decomposition \eqref{eq:decomp} and write the approximation as
\[
u_{\e}^{\Delta_{eXp}} = u_0^{\Delta_{eXp}} + u^{BL, \Delta_{eXp}}_{\e,-} + u^{BL,\Delta_{eXp}}_{\e,+} + u_{\e}^{R,\Delta_{eXp}}\;.
\]

\medskip
\noindent\textbf{Boundary layer terms $u^{BL}_{\e,\pm}$.}
By symmetry, we only analyze the error for the layer at the left endpoint. 
As in the case of Proposition \ref{prop:shish}, we choose $\theta = 2 \beta_{*}$. 
For each $i=1,\dots,N/4$, it holds that
\begin{eqnarray*}
\Vert  \left( u^{BL}_{\e,-} \right)'' \Vert^{2}_{L^2 (I_{i})} 
&=& \int_{x_{i-1}}^{x_i} \left( \left( u^{BL}_{\e,-} (x) \right)'' \right)^2  dx \\
&\leq& \frac{1}{\min_{x \in I_i } \left( w_{\e}^{-} (x) \right)^2}  \int_{x_{i-1}}^{x_{i}} \left( \left( u^{BL}_{\e,-} (x)\right)'' \right)^2  \left( w_{\e}^{-} (x) \right)^2 dx \\
&=& \frac{1}{\left( w_{\e}^{-} (x_{i-1}) \right)^2} \Vert  \left( u^{BL}_{\e,-} \right)'' \Vert^{2}_{L^2_{w_{\e}}(I_{i})} \\
&=& \psi\left(\frac{2(i-1)}{N}\right)^4 \Vert  \left( u^{BL}_{\e,-} \right)'' \Vert^{2}_{L^2_{w_{\e}}(I_{i})}\;,
\end{eqnarray*}
where we have used the monotonicity of $w_{\e}^{-}$.

%The Bramble-Hilbert Lemma (see, e.g., \cite[Section 4.4]{BrennerScott2008}) for $k=0,1$
In the following we denote by $\mathcal{I}(u^{BL}_{\e,-})$ the CpwL nodal interpolant of $u^{BL}_{\e,-}$ on the eXp mesh. 
Using classical interpolation error estimates 
and \eqref{eq:hj_exp_psi}, we obtain, for $\ell=0,1$,
\begin{align*}
&\Vert  \left( u^{BL}_{\e,-} \right)^{(\ell)} - \left( \mathcal{I}(u^{BL}_{\e,-}) \right)^{(\ell)}\Vert^{2}_{L^2 (I_{{\text{layer}}}^{-})} \leq C \sum_{i=1}^{N/4} h_i^{2(2-\ell)} \Vert  \left( u^{BL}_{\e,-} \right)'' \Vert^{2}_{L^2 (I_{i})}\\
&\leq C \sum_{i=1}^{N/4} \left(\e \theta \ln\left(\frac{\psi\left(\frac{2(i-1)}{N}\right)}{\psi \left(\frac{2(i-1)}{N}\right)- \delta} \right)\right)^{2(2-\ell)} \psi\left(\frac{2(i-1)}{N}\right)^{4} \Vert  \left( u^{BL}_{\e,-} \right)'' \Vert^{2}_{L^2_{w_{\e}}(I_{i})}\\
&\leq C \e^{2(2-\ell)} \max_{i=1,\dots,N/4} \left(\psi\left(\frac{2(i-1)}{N}\right)^{4} \ln\left(\frac{\psi\left(\frac{2(i-1)}{N}\right)}{\psi \left(\frac{2(i-1)}{N}\right) - \delta} \right)^{2(2-\ell)} \right) \sum_{i=1}^{N/4} \Vert  \left( u^{BL}_{\e,-} \right)'' \Vert^{2}_{L^2_{w_{\e}}(I_{i})}\\
&\leq C \e^{2(2-\ell)} \max_{i=1,\dots,N/4} \left(\psi\left(\frac{2(i-1)}{N}\right)^{4} \ln\left(\frac{\psi\left(\frac{2(i-1)}{N}\right)}{\psi \left(\frac{2(i-1)}{N}\right) - \delta} \right)^{2(2-\ell)} \right) \Vert  \left( u^{BL}_{\e,-} \right)'' \Vert^{2}_{L^2_{w_{\e}}(I)}\;.
\end{align*}

We now claim that there exists $C>0$ independent of $N, \e$ such that
\begin{equation}\label{eq:claim}
\max_{i=1,\dots,N/4} \left(\psi\left(\frac{2(i-1)}{N}\right)^{4} \ln\left(\frac{\psi\left(\frac{2(i-1)}{N}\right)}{\psi \left(\frac{2(i-1)}{N}\right) - \delta} \right)^{2(2-\ell)} \right) \leq C N^{-2(2-\ell)}.
\end{equation}
Indeed, for $i=1,\dots,N/4$, we have that 
\[
\frac{6N - 8}{N^2} \leq \psi\left(\frac{2(i-1)}{N}\right) \leq 1\;.
\]
Consider the map 
\[
g_{\ell} \colon y \mapsto  y^4 \ln\left(\frac{y}{y - \delta}\right)^{2(2-\ell)}, \quad y \in \left[\frac{6N - 8}{N^2} , 1\right]\;.
\]
Then, for $\ell=0$,  
we note that finding the maximum of $g_{0}$ in $\left[\frac{6N - 8}{N^2} , 1\right]$ 
is equivalent to finding the maximum of $h \colon y \mapsto y \ln\left(\frac{y}{y - \delta}\right)$. 
We have that
\[
h'(y) = \ln\left(1 + \frac{\delta}{y - \delta}\right) - \frac{\delta}{y - \delta} < 0,\quad \text{for } y \in \left[\frac{6N - 8}{N^2} , 1\right]\;,
\]
where we used that for $t > -1$, $\ln(1 + t) < t$. 
Therefore, $g_{0}$ attains its maximum at the left endpoint. 
Moreover, $\frac{6N - 8}{N^2} - \delta = \frac{2}{N}$, and hence
\[
\frac{6N - 8}{N^2} \ln\left(3 - \frac{4}{N}\right) \leq \frac{6}{N} \ln(3)\;.
\]
This shows the claim in the case $\ell=0$.

For $\ell=1$, we have that
\[
g_{1}'(y) = y^3 \ln\left(\frac{y}{y - \delta}\right) \left[4 \ln\left(\frac{y}{y - \delta}\right) - \frac{2 \delta}{y - \delta}  \right]\;.
\]
The sign of $g_{1}'(y)$ is determined by
\[
\Phi(y) = 2 \ln\left(\frac{y}{y - \delta}\right) - \frac{\delta}{y - \delta} = 2\ln(1+u) - u\;,
\]
for $u = \frac{\delta}{y - \delta} > 0$. From this we obtain that $\Phi(y) \geq 0$, for all $y \in \left[\frac{6N - 8}{N^2} , 1\right]$. 
It follows that $g_{1}$ is monotonically increasing and attains its maximum at $y=1$. Finally, using that $- \ln(1 - z) \leq \frac{z}{1 - z}$ for $0 \leq z < 1$ and
$\frac{4}{N} - \frac{8}{N^2} \leq \frac{1}{2}$, $N \geq 4$, we find
\[
- \ln\left(1 - \frac{4}{N} + \frac{8}{N^2}\right) \leq \frac{\frac{4}{N} - \frac{8}{N^2}}{1 - \frac{4}{N} + \frac{8}{N^2}} \leq 2 \left(\frac{4}{N} - \frac{8}{N^2} \right)\leq \frac{8}{N}\;,
\]
and the claim follows. 
This shows the result for $u_{\e,-}^{BL}$ on $I_{\text{layer}}^{-}$. 

On $\tilde{I} = I \setminus I_{\text{layer}}^{-}$, the 
same argument as in the proof of Proposition \ref{prop:shish} can be used.

\medskip
\noindent\textbf{Smooth part $u_0$ and remainder $u_{\e}^R$.}
As in Proposition \ref{prop:shish}, we choose the approximation of the remainder, 
$u_{\e}^{R,\Delta_{eXp}}$, to be zero.

We next consider the smooth part $u_0$. 
Let $u_0^{\Delta_{eXp}}$ denote the CpwL interpolant of $u_0$ on the eXp mesh.
On $I_{\mathrm{coarse}}$, arguing as in Proposition~\ref{prop:shish} and using in
particular \eqref{eq:rate_smooth}, we obtain
$$
\Vert u_0 - u_0^{\Delta_{{eXp}}} \Vert_{1,\e,I_{\text{coarse}}} 
\leq C N^{-1}.
$$
For the layer region, by symmetry it suffices to estimate the error on $I_{\mathrm{layer}}^-$. 
Using the classical $L^2$ interpolation error estimate
together with the stability of CpwL interpolation in the $H^1$-seminorm, 
we obtain
$$
\Vert u_0 - u_0^{\Delta_{{eXp}}} \Vert^2_{1,\e,I_{\text{layer}}^{-}} 
\leq 
C \sum_{i=1}^{N/4} \left( \e^2 \Vert  u'_{0} \Vert^2_{L^2(I_i)}+h_i^{2} \Vert u_{0}' \Vert^2_{L^2(I_i)} \right).
$$
Since by assumption $\e \leq N^{-1}$ and, by \eqref{eq:hj_exp}, 
$h_i \leq C N^{-1}$, for $i=1,\dots,N/4$, it follows that
\begin{eqnarray*}
\Vert u_0 - u_0^{\Delta_{{eXp}}} \Vert^2_{1,\e,I_{\text{layer}}^{-}}
 \leq C N^{-2} \sum_{i=1}^{N/4}   \Vert u_{0}' \Vert^2_{L^2(I_i)} 
 \leq C N^{-2},
\end{eqnarray*}
where \eqref{eq:u0H2} was used in the last step.

Finally, we mention that by construction the approximant satisfies the zero Dirichlet BCs.
\end{proof}
%%%%%%%%%%%%%%%%%%%%%%%%%%%%%%%%%%%%%%%%%%%%%%%%%%%%%%%%%%%
%%%%%%%%%%%%%%%%%%%%%%%%%%%%%%%%%%%%%%%%%%%%%%%%%%%%%%%%
\subsubsection{Bakhvalov-Shishkin mesh}
%%%%%%%%%%%%%%%%%%%%%%%%%%%%%%%%%%%%%%%%%%%%%%%%%%%%%%%%
In the original Bakhvalov mesh \cite{Bakhvalov1969} the mesh points are chosen by solving a non-linear problem. 
This results in a layer-adapted mesh with more nodal points in the layer regions, and usually far fewer points in the coarse region. A combination of the above two meshes, referred to as the Bakhvalov-Shishkin (B-S) mesh, includes a uniform subdivision in the coarse region and the (original) Bakhvalov nodal points in the layer regions (see, e.g., \cite{Linss1985}). For the B-S mesh, we have 
$$
\phi(t) = - \ln\left( 1 - 2t (1-N^{-1}) \right) \; , \; \psi(t) = 1 - 2t(1-N^{-1}),
$$
and there holds $\max|\psi'| = 2$. 
We denote this mesh as $\Delta_{BS}$.

In numerical experiments (see, e.g., \cite{FranzCAX2016}) the eXp mesh 
outperforms both the Shishkin and B-S meshes in the practical range of $N$ and $\e$.
%even though this is not obvious from the theory.
An analog of Proposition \ref{prop:eXp} holds for the B-S mesh (with an almost identical proof).
%For brevity, we omit it.
%
%%%%%%%%%%%%%%%%%%%%%%%%%%%%%%%%%%%%%%%%%%%%%%%%%%%%%%%%%%%
\section{Robust Neural Network Approximation}
\label{sec:relunn}
%%%%%%%%%%%%%%%%%%%%%%%%%%%%%%%%%%%%%%%%%%%%%%%%%%%%%%%%%%%
In this section, we consider the approximation of univariate functions on bounded intervals
by NNs with $\ReLU^k$ activation function
$\varrho: \R \to \R: x \mapsto \max\{0, x\}^k$, \; $k=1,2$.
The case $k=1$ corresponds to the classical ReLU activation.
It allows to exactly represent continuous, piecewise affine
functions on arbitrary partitions in $\R$.
Owing to $\ReLU$ NNs not allowing for exact multiplication 
of real numbers, higher degree polynomials and splines can 
not be exactly represented by $\ReLU$ NNs, motivating to
consider $k\geq 2$.
%%%%%%%%%%%%%%%%%%%%%%%%%%%%%%%%%%%%%%%%%%%%%%%%%%%%%%%%%%%
\subsection{Neural Networks}
\label{sec:nn}
%%%%%%%%%%%%%%%%%%%%%%%%%%%%%%%%%%%%%%%%%%%%%%%%%%%%%%%%%%%
As usual (e.g. \cite{OPS2020,OS2023,PV2018}), 
we define a neural network (NN) in terms of its 
weight matrices $A_\ell$ and bias vectors $b_\ell$.

\emph{
All NNs considered are feedforward NNs without skip connections, 
and we shall not indicate this at each occurrence. 
}
We distinguish between a 
NN configuration $\Phi$ of a feedforward NN 
and the function it realizes,
called \emph{realization} of the NN $\Phi$,
and denoted as $\realiz(\Phi)$ being
the composition of parameter-dependent affine transformations
and nonlinear activations.
We recall some NN formalism in the notation of
\cite[Section 2]{PV2018}.
%%%%%%%%%%%%%%%%%%%%%%%%%%%%%%%%%%%%%%%%%%%%%%%%%%%
\begin{definition}[{\cite[Definition 2.1]{PV2018}}]{[Neural Network Definitions]}
\label{def:NeuralNetworks}
\newline

\noindent
{1.}
For $d,L\in\N$, a \emph{neural network configuration $\Phi$} 
with input dimension $d \geq 1$ and number of layers $L\geq 1$, 
comprises 
a finite sequence of matrix-vector pairs, i.e.
\begin{align*}
\Phi = ((A_1,b_1),(A_2,b_2),\ldots,(A_L,b_L)).
\end{align*}

\noindent
{2.}
For $N_0 := d$ and \emph{numbers of neurons $N_1,\ldots,N_L\in\N$ per layer}, 
for all $\ell=1,\ldots, L$ it holds that
$A_\ell\in\R^{N_\ell \times N_{\ell-1} }$ and
$b_\ell\in\R^{N_\ell}$.

\noindent
{3.} [Realization]
For a NN $\Phi$ and an activation function $\varrho: \R \to \R$, 
we define the associated
\emph{realization of feedforward NN $\Phi$} 
as the function
\begin{align*}
\realiz(\Phi): \R^d\to\R^{N_L} : x \to x_L,
\end{align*}
where, for $L\geq 2$,
\begin{align*}
x_0 & := x,
\\
x_\ell & := \varrho( A_\ell x_{\ell-1} + b_\ell ),
\qquad\text{ for }\ell=1,\ldots,L-1,
\\
x_L & := A_L x_{L-1} + b_L.
\end{align*}

\noindent
{4.}[Activation]
The \emph{activation function} $\varrho$ 
is assumed to act componentwise on vector-valued inputs, 
$\varrho(y) = (\varrho(y_1), \dots, \varrho(y_m))$ 
for all $y = (y_1, \dots, y_m) \in \R^m$.

\noindent
{5.}
The layers indexed by $\ell=1,\ldots,L-1$ 
are \emph{hidden layers}.
In these layers the activation function is applied.
No activation is applied in the last 
layer $L$ of the NN.

\noindent
{6.} [Width, Depth, Size]
We refer to 
$\depth(\Phi) := L$ as the \emph{depth} of $\Phi$
and call 
$\size(\Phi) := \sum_{\ell=1}^L \norm[0]{A_\ell} + \norm[0]{b_\ell}$
the \emph{size} of $\Phi$,
which is the number of nonzero components 
in the weight matrices $A_\ell$ and the bias vectors $b_\ell$.
The \emph{number of neurons} in the NN $\Phi$ is denoted by $\#(\Phi)$,
and the \emph{ width} of the NN $\Phi$ of depth $L(\Phi)$ 
is $W(\Phi) = \max_{\ell=1,...,L(\Phi)} N_\ell$. 
Without additional assumptions on the sparsity of $b_\ell$ and $A_\ell$, 
$M(\Phi) \lesssim (\#(\Phi))^2$.

\noindent
{7.} [Input / Output Dimension]
The parameters 
$d, N_L \in \N$ are the \emph{input dimension} and the \emph{output dimension},
respectively.

\noindent
{8.}
The set of all such NNs is denoted by 
$\mathcal{NN}^\varrho_{L,W,d,N_L}$.
\end{definition}
\begin{remark}\label{rmk:NNsize}
Some related works, e.g. \cite{DLM2021}, 
use the width as a measure for the complexity of a NN.
In each layer of a fully connected NN
the number of nonzero weights can be width squared.
\end{remark}
\begin{remark}\label{rmk:StrctNN}
We refer to NNs with only activation function $\varrho$ 
as \emph{strict $\varrho$-NNs}, or simply as $\varrho$-NNs.
This includes NNs of depth $L = 1$, 
which do not have hidden layers and 
which exactly realize affine transformations.
\end{remark}
\begin{remark}\label{rmk:MultiActi}[Multiple Activations]
NNs $\Phi$ of depth $L=\depth(\Phi) \geq 2$ 
may admit a different activation $\rho_\ell$, $\ell = 1,...,L-1$ 
at each hidden layer. We will leverage this to 
include exponential boundary layers via $\tanh( )$ and 
sigmoidal activations into the feature space.
In the case of heterogeneous activations, 
we write 
$\bsvarrho := \{ \varrho_\ell: \ell =1,...,L-1\}$,
and denote the NN as $(\Phi,\bsvarrho)$,
with realization $\realiz(\Phi,\bsvarrho)$.
In emulation of boundary layers via $\tanh$-activations,
only parts of NN layers are activated with $\tanh$, 
while
the remaining parts are $\ReLU$ or $\ReLU^2$ activated.
We then collect these activations also in $\bsvarrho$, 
and denote with $\realiz(\Phi,\bsvarrho)$ the realization
obtained with activations distributed within the NN architectures
as will be clear from the context.
\end{remark}
%%%%%%%%%%%%%%%%%%%%%%%%%%%%%%%%%%%%%%%%%%%%%%%%%%%%%%%%%%%
\subsection{ReLU NN Emulation of continuous, piecewise polynomials}
\label{sec:ReLU CpwL}
%%%%%%%%%%%%%%%%%%%%%%%%%%%%%%%%%%%%%%%%%%%%%%%%%%%%%%%%%%%
A first (well-known) observation is that
on arbitrary, finite partitions $\Delta$ of the bounded interval $(a,b)$, 
the space ${\mathcal S}^1(\Delta)$ of continuous, piecewise linears can be 
exactly represented by shallow ReLU NNs.
We refer to  \cite[Prop.~3.4]{AHS23_2956} for a proof.
%%%%%%%%%%%%%%%%%%%%%%%%%%%%%%%%%%%%%%%%%%%%%%%%%%%%%%%%%%%%%%%%%%%%%%%%%
\begin{proposition}\label{prop:approx_pl_func_ReLUNN}
Let $N\in \mathbb{N}$, and 
$\Delta = \{ x_j \}_{j=0}^N$ a partition of $I$ into $N$ 
open, disjoint, connected subintervals
$I_j = (x_{j-1},x_j)$ of length $h_j = x_j - x_{j-1}$, $j=1,\ldots,N$,
with mesh width $h = \max_{j=1}^N h_j$.

For each $\phi_N \in \mathcal{S}^{1}(\Delta)$
there is a ReLU neural network
$\Phi_N\in \mathcal{NN}_{2,{N+1},1,1}$ 
such that
$\phi_N(t) = \realiz(\Phi_N)(t)$ for all $t\in I$.
Its weights are
bounded in absolute value by 
$\max \left\{1+h_1,\|\phi_N\|_{L^\infty(I)},2\displaystyle\max_{j=0,\dots,N}\frac{1}{h_j}\right\}$.
\end{proposition}
%%%%%%%%%%%%%%%%%%%%%%%%%%%%%%%%%%%%%%%%%%%%%%%%%%%%%%%%%%%%%%%%%%%%%%%%%
\subsection{Robust Shallow $\ReLU$-NN Solution Approximation}
\label{sec:RobNNSolAppr}
%%%%%%%%%%%%%%%%%%%%%%%%%%%%%%%%%%%%%%%%%%%%%%%%%%%%%%%%%%%%%%%%%%%%%%%%%
With Prop.~\ref{prop:approx_pl_func_ReLUNN}, 
Propositions \ref{prop:Reg1}, \ref{prop:shish}, \ref{prop:eXp} 
immediately imply robust $\ReLU$-NN emulation results
for shallow, fixed-depth $\ReLU$ NNs.
\begin{theorem}
\label{thm:relubalanced}
Suppose the assumptions of Prop.~\ref{prop:Reg1} hold. 
For $\e\in(0,1]$, $I=(-1,1)$
let $u_{\e}\in H^1_0(I)$ be the solution of (\ref{eq:BVPweak}).

Then, 
for every $0<\e\leq 1$ and for $1\leq N \leq 1/\e$ 
exist shallow ReLU NNs 
$\{\Phi^{S,N}_{\e}\}_{0<\e\leq 1},
 \{\Phi^{\text{eXp},N}_{\e}\}_{0<\e\leq 1} \in \mathcal{NN}_{2,{N+1},1,1}$
and a constant $C>0$ independent of $\e \in (0,1]$
such that
\begin{align}
\label{eq:reluShish}
\e \|u_{\e}' - \realiz(\Phi^{S,N}_{\e})'\|_{L^2(I)} 
	+ \|u_{\e} - \realiz(\Phi^{S,N}_{\e}) \|_{L^2(I)} 
\leq &\, C N^{-1}(1+\log(N)),
\\
\label{eq:relueXp}
\e \|u_{\e}' - \realiz(\Phi^{\text{eXp},N}_{\e})'\|_{L^2(I)} 
        + \|u_{\e} - \realiz(\Phi^{\text{eXp},N}_{\e}) \|_{L^2(I)} 
\leq &\, C N^{-1},
\end{align}
and 
\[
\realiz(\Phi^{S,N}_{\e})(\pm1) = 0, \quad \realiz(\Phi^{\text{eXp},N}_{\e})(\pm1) = 0.
\]

Furthermore, 
$\depth(\Phi^{S,N}_{\e}) = \depth(\Phi^{\text{eXp},N}_{\e}) = 2$ 
and 
there is a constant $\tilde{C}>0$ independent of $\e$ and of $N$ 
such that
\begin{equation}
\label{eq:relublsize}
\size\left(\Phi^{S,N}_{\e}\right) + \size\left(\Phi^{\text{eXp},N}_{\e}\right) 
\leq \tilde{C} N
\;.
\end{equation}
The weights and biases in the hidden layer of $\Phi^{S,N}_{\e}$ 
are independent of $u_{\e}$ and depend only on $\e$ and $N$.
\end{theorem}

\begin{proof}
We invoke the Shishkin approximation rate bound \eqref{eq:ShiskBd},
and Prop.~\ref{prop:approx_pl_func_ReLUNN} 
stating that shallow ReLU NNs represent 
any element $v\in \mathcal{S}^1(\Delta)$ \emph{exactly},
on \emph{any partition} $\Delta$ of $I$.
We choose $N$ as in the bound \eqref{eq:ShiskBd}, 
for $N$ and $0 <\e \leq 1$, $\Delta$ as the Shishkin mesh $\Delta_S$, 
and leverage Prop.~\ref{prop:approx_pl_func_ReLUNN}.
Any element $v\in \mathcal{S}^1(\Delta_S)$ can be exactly represented
by (realization of) a NN $\realiz(\Phi^{S,N}_{\e})$ of depth $2$,
and width $O(N)$.

For the exponential mesh $\Delta_{eXp}$, we argue in the same way,
using Prop.~\ref{prop:eXp}.
\end{proof}

Similar to the proof of Theorem~\ref{thm:relubalanced}, other known robust 
approximation and FE convergence rate bounds have ReLU NN analogs. We
briefly comment on these, based on \cite[Thms. 5.1-5.3]{SunStynes95}.
Problem \eqref{eq:de}, \eqref{eq:bc} is a special case of \cite[Eqn.~(1.1)]{SunStynes95},
with (notation as in \cite{SunStynes95}) 
$m=1$, $L_1 \equiv 0$, and $a_0(x) = b(x)$.
Based on the results in \cite[Section~5]{SunStynes95} and using 
Prop.~\ref{prop:approx_pl_func_ReLUNN} we obtain 
\begin{proposition}\label{prop:SSCor5.1}
Assume that $b,f \in C^2(\overline{I})$. 

Then for every $2 \leq N\in \N$ and for every $\e \in (0,1]$, 
there is a ReLU NN of width $O(N)$ such that 
for $\{ u_\e \}_{0 < \e \leq 1}$ solution of \eqref{eq:de}, \eqref{eq:bc}  
holds
$$
\inf_{\Phi \in \mathcal{NN}^{\ReLU}_{2,N,1,1}} \| u_{\e} -  \realiz(\Phi) \|_{1,\e,I}
\leq 
C N^{-1}\log(N) 
\;,
$$
and
$$
\inf_{\Phi \in \mathcal{NN}^{\ReLU}_{2,N,1,1}} \| u_{\e} -  \realiz(\Phi) \|_{L^2(I)} 
\leq 
C(N^{-1}\log(N))^{2}
\;.
$$
Here, the constant $C>0$ is independent of $\e$ and $N$.
\end{proposition}
\begin{proof} 
The proof is based on \cite[Cor.~5.1]{SunStynes95}, Eqn.~(5.20).
With our choices $m=1$, we find that with Prop.~\ref{prop:approx_pl_func_ReLUNN} 
the error bounds in \cite[Cor.~5.1]{SunStynes95}
for continuous, piecewise affine approximations of $u_\e$ translate into the 
above bounds for the best approximations with shallow ReLU NNs.
\end{proof}
We emphasize that Prop.~\ref{prop:SSCor5.1} is an \emph{approximation error bound} 
for $u_\e$, solution of \eqref{eq:de}, \eqref{eq:bc}, 
which is obtained by CpwL functions on a Shishkin mesh $\Delta_S$.

For the Galerkin FE approximation $u^{\Delta_S}_\e \in S^1_0(\Delta_S)$ 
the corresponding robust error bounds from \cite[Chap.~5.3]{SunStynes95}
likewise translate to shallow DNN approximations, provided that the 
\emph{DNN loss function is chosen as, e.g. in the deep Ritz method}  \cite{DeepRitz}
corresponding to the ``energy'' functional
$$
{\mathcal R}_\e (v) := \frac{1}{2} B_\e(v,v) - F(v) \;,\; v\in H^1_0(I) \;.
$$
Here, the definition of ${\mathcal R}_\e$ involves 
suitable numerical integration evaluations of $b(x)$ and $f(x)$.

In this case, 
shallow DNN approximations $u^{NN}_\e$ are obtained as
minimizers of ${\mathcal R}_\e$, i.e.
$$
u^{NN}_\e := {\rm arg}\min 
\left\{ {\mathcal R}_\e (v^{NN}) : v^{NN} \in \realiz(\mathcal{NN}^{\ReLU}_{2,N,1,1}) \right\}
\;.
$$
Based on \cite[Cor.~5.2]{SunStynes95}, one has 
$$
\inf_{v^{NN} \in \realiz(\mathcal{NN}^{\ReLU}_{2,N,1,1})}
\| u_\e - v^{NN} \|_{1,\e,I} \leq C N^{-1}\log(N) \;.
$$
%
%%%%%%%%%%%%%%%%%%%%%%%%%%%%%%%%%%%%%%%%%%%%%%%%%%%%%%%%%%%%%%%%%%%%%%%%%
\subsection{Robust Deep $\ReLU$ Super-Approximation}
\label{sec:SupReLU}
%%%%%%%%%%%%%%%%%%%%%%%%%%%%%%%%%%%%%%%%%%%%%%%%%%%%%%%%%%%%%%%%%%%%%%%%%
The preceding result, Thm.~\ref{thm:relubalanced},
addressed the robust approximation of the solution $u_\e\in H^1_0(I)$ 
of the model problem in Sec.~\ref{sec:model},
by \emph{shallow} $\ReLU$-NNs, 
under the low regularity assumptions on the data $b(x)$ and $f(x)$ 
in Prop.~\ref{prop:Reg1}.
We now consider robust approximation rate bounds for deep NNs, 
still under the low Sobolev regularity on $f$, $b$ in \eqref{eq:de} 
which we assumed in Prop.~\ref{prop:Reg1}.

The principle of proof of Theorem~\ref{thm:relubalanced} 
was to leverage (known) $\e$-robust convergence
rate bounds of CpwL FE/FD approximations and to interpret the CpwL 
approximations on $\e$-dependent, Shishkin meshes 
in these results as \emph{exact realizations} of 
$\ReLU$-NNs of a suitable (fixed depth) architecture, 
using Prop.~\ref{prop:approx_pl_func_ReLUNN}.

In the present section, we go beyond this type of result. 
While still assuming finite, low Sobolev regularity of the data $b(x)$ and $f(x)$ 
in \eqref{eq:de}, we admit
heterogeneous activations for more efficient emulation of boundary layers, 
and, we leverage depth to obtain higher order convergence rate bounds. 
Specifically,
\emph{for the data-regularity given in Prop.~\ref{prop:Reg1},
certain deep feedforward NNs achieve robust rates of convergence 
that are higher, in terms of the network size,  
than those afforded by the CpwL FE approximations on eXp and Shishkin meshes.}
Our arguments are again based on the solution decomposition 
\eqref{eq:decomp} in Prop.~\ref{prop:Reg1}.
On grounds of the (assumed) constant reaction term,
the boundary layers are exponential (cf. Cor.~\ref{cor:bconst}).
We observe in Section~\ref{sec:NNexpBL}
that the use of $\tanh$ and sigmoid activations 
creates an ``exponential boundary layer feature space''.
Furthermore, 
in Sections~\ref{sec:sReLUSmFct}, \ref{sec:sReLUue} 
we leverage depth in the form of
so-called ``superconvergent'' expression rate bounds on Sobolev classes
for $\ReLU$-NNs shown e.g. in \cite{li2025superapproximationratesreluneural,yang2025deepneuralnetworksgeneral}. 
Unlike the fixed-depth NNs in Sec~\ref{sec:RobNNSolAppr},
super-expressivity results for these $\ReLU$ architectures entail 
networks whose depth $L$ is proportional to width, as explained in
\cite{yang2025deepneuralnetworksgeneral}.

We develop expression rate results in two scenarios: 
(i) the general case of variable reaction coefficient $b(x)$ in Prop.~\ref{prop:Reg1}, 
and 
(ii) the particular case of constant reaction coefficient in Cor.~\ref{cor:bconst},
where the boundary layer functions $u^{BL}_{\e,\pm}(x)$ 
are exponential functions \eqref{eq:uBLexp}.
We start by reviewing expression rate bounds for these, from \cite{OSX24_1085}.
%%%%%%%%%%%%%%%%%%%%%%%%%%%%%%%%%%%%%%%%%%%%%%%%%%%%%%%%%%%%%%%%%%%%%%%%%
\subsubsection{$\tanh$ NN emulation of exponential boundary layers}
\label{sec:NNexpBL}
%%%%%%%%%%%%%%%%%%%%%%%%%%%%%%%%%%%%%%%%%%%%%%%%%%%%%%%%%%%%%%%%%%%%%%%%%
In the model problem \eqref{eq:de}, \eqref{eq:bc} 
\emph{with constant reaction coefficient} $b(x)$, 
the boundary layers $u^{BL}_{\e,\pm}(x)$ 
are exponentials of scaled and shifted exponentials, cf. Cor.~\ref{cor:bconst}.
These allow \emph{exact, parsimonious NN emulation} with suitable activations.
The following result, proved in \cite[Sec.~7]{OSX24_1085},
shows that using $\tanh(\cdot)$- %and sigmoidal 
activations includes exponential boundary layer functions 
$u^{BL}_{\e,\pm}$ as in \eqref{eq:uBLexp}
in the NN feature space.
\begin{lemma}
\label{lem:tanhnnexp} 
[$\tanh$-emulation of exponential boundary layers \cite[Lem.~7.2]{OSX24_1085}]
For all $\tau\in(0,1]$ there exists a $\tanh$ NN $\Phi^{\exp}_{\tau}$
such that for all $x\geq 0$ there holds
\begin{subequations}
\label{eq:tanhnnexperr} 
\begin{align}
\label{eq:tanhnnexperri}
\snorm{ \exp(-x) - \realiz( \Phi^{\exp}_{\tau} )(x) }
        \leq &\, \exp(-\tau),
        \\
        \label{eq:tanhnnexperrii} 
\snorm{ -\exp(-x) - \realiz( \Phi^{\exp}_{\tau} )'(x) }
        \leq &\, \exp(-\tau),
\end{align}
\end{subequations} 
and such that 
$\depth( \Phi^{\exp}_{\tau} ) = 2$
and
$\size( \Phi^{\exp}_{\tau} ) = 4$.
\end{lemma}
According to \cite[Remk.~7.3]{OSX24_1085}, 
such a result also holds for the sigmoidal activation 
\[
\varrho(x) = \tfrac{\exp(x)}{1 + \exp(x)} = \tfrac{1}{1 + \exp(-x)},
\qquad
x\in\R.
\]
%
%%%%%%%%%%%%%%%%%%%%%%%%%%%%%%%%%%%%%%%%%%%%%%%%%%%%%%%%%%%%%%%%%%%%%%%%%
\subsubsection{Robust Exponential Deep $\ReLU$ NN approximation of Layers}
\label{sec:ReLUexpBL}
%%%%%%%%%%%%%%%%%%%%%%%%%%%%%%%%%%%%%%%%%%%%%%%%%%%%%%%%%%%%%%%%%%%%%%%%%
If only $\ReLU$-activations are admitted, 
in \cite[Prop.~5.3]{OSX24_1085}
rates of expression for strict $\ReLU$ NNs of 
exponential boundary layers are shown.
\begin{lemma} \label{lem:reluexp} 
There exist positive constants $\beta, C> 0$ such that 
for every $\e \in (0,1]$, $p\in \N$ there are $\ReLU$ NNs 
$\Phi^{\exp,p}_\e \in \mathcal{NN}^{\ReLU}_{L,N,1,1}$ 
such that
it holds
\begin{equation}\label{eq:ReLUBL1}
\e | \exp(\cdot/\e) - \realiz(\Phi^{\exp,p}_\e) |_{H^1(I)} 
\leq C\exp(- \beta p) \;,
\end{equation}
and, for $q\in \{2,\infty\}$,
\begin{equation}\label{eq:ReLUBL0} 
\| \exp(\cdot/\e) - \realiz(\Phi^{\exp,p}_\e) \|_{L^q(I)} 
\leq C\exp(- \beta p) \;.
\end{equation}
Depth and size of the NNs $\Phi^{\exp,p}_\e$ 
are bounded as 
\begin{equation}\label{eq:LN}
L(\Phi^{\exp,p}_\e) \leq C p (1+\log_2(p)) \;,
\quad 
M(\Phi^{\exp,p}_\e) \leq C p^2 \;.
\end{equation}
\end{lemma}
We remark that for a target approximation accuracy $0 < \tau \leq 1$
in the norm $\| \circ \|_{1,\e,I}$,
based on \eqref{eq:ReLUBL0}, \eqref{eq:ReLUBL1} one chooses $p= O(|\log(\tau)|)$. 
This results in \eqref{eq:LN} in the $\e$-uniform depth and size bounds
\begin{equation}\label{eq:LNtau}
L(\Phi^{\exp,p}_\e) \leq C |\log(\tau)| (1+\log(|\log(\tau)|)) \;,
\quad 
M(\Phi^{\exp,p}_\e) \leq C |\log(\tau)|^2 
\;.
\end{equation}
%
%%%%%%%%%%%%%%%%%%%%%%%%%%%%%%%%%%%%%%%%%%%%%%%%%%%%%%%%%
\subsubsection{$\ReLU$ NN Super-Approximation of smooth functions}
\label{sec:sReLUSmFct}
%%%%%%%%%%%%%%%%%%%%%%%%%%%%%%%%%%%%%%%%%%%%%%%%%%%%%%%%%
For neural approximation bounds for the 
smooth solution parts $u_0$ and $u^R_\e$ 
in the decomposition \eqref{eq:decomp},
we recap recent results from \cite[Thms.~1 \& 2]{li2025superapproximationratesreluneural} 
on deep ReLU NN expression rates for smooth Sobolev functions.
In these results, bitstring encodings afford higher convergence rates 
in Sobolev spaces 
in terms of the NN size (expressed in terms of the number of neurons)
than the best piecewise polynomial (spline) approximations.
We require only a 
particular case of \cite[Thms.~1 \& 2]{li2025superapproximationratesreluneural}
($d=1$, $p=2$, $m=0,1$), which reads as follows.
\begin{proposition}\label{prop:ReLUSupr} 
{[$\ReLU$ super approximation in $L^2(I)$ and $H^1(I)$]}

In $I = (0,1)$, for $q\in [1,\infty]$,
let $w\in (W^{2,q}\cap W^{1,q}_0)(I)$ be given.

Then there exists $C>0$ such that 
for any $L\in \N$ and $W\in \N$, 
exists a NN $\Phi^w_{L,W} \in \mathcal{NN}^{\ReLU}_{\bar{L},\bar{W},1,1}$ 
such that for $k=0,1$
\begin{equation}\label{eq:sReLUBd}
| w - \realiz(\Phi^w_{L,W} ) |_{W^{k,q}(I)} 
\leq 
C (W L)^{-2(2-k)} (\log(2WL))^3 | w |_{W^{2,q}(I)} \;.
\end{equation}
In \eqref{eq:sReLUBd}, 
Depth $\bar{L} = \bar{L}(\Phi^w_{L,W})$ 
and 
width $\bar{W} = \bar{W}(\Phi^w_{L,W})$ 
are bounded as
\begin{equation}\label{eq:sReLUdw}
\bar{L} \leq C L \log(8L)^2 \;, 
\quad 
\bar{W} \leq C W \log(8W)^2 \;.
\end{equation}
\end{proposition}
We remark that in terms of the NN size of $\Phi^w_{L,W}$, 
see Def.~\ref{def:NeuralNetworks}, Item. 6,
$$
\#(\Phi^w_{L,W}) 
\leq C \bar{L}\bar{W} 
\leq C LW \log(8L)^2 \log(8W)^2 \;.
$$
and
$$
\size(\Phi^w_{L,W}) 
\leq C \bar{L}\bar{W}^2 
\leq C LW^2 \log(8L)^2 \log(8W)^4 \;.
$$ 
We shall use these rates for the $\e$-uniform approximation of the ``smooth'' parts 
$u_0$ and $u_\e^R$ in the decomposition \eqref{eq:decomp} while 
leveraging Lemma~\ref{lem:tanhnnexp} resp. 
the exponential robust boundary layer approximation rate bound 
\cite[Prop.~5.3]{OSX24_1085} by strict $\ReLU$ NNs.
%%%%%%%%%%%%%%%%%%%%%%%%%%%%%%%%%%%%%%%%%%%%%%%%%%%%%%%%%55
\subsubsection{Robust Deep NN Expression Rates for $u_\e$}
\label{sec:sReLUue}
%%%%%%%%%%%%%%%%%%%%%%%%%%%%%%%%%%%%%%%%%%%%%%%%%%%%%%%%%55
With the preceding results, we are in position to 
establish existence of $\e$-robust, super-convergent 
NN approximation rate bounds of solutions 
$u_\e \in H^1_0(I)$ of \eqref{eq:de}, \eqref{eq:bc}.
Then Prop.~\ref{prop:Reg1} furnishes the decomposition \eqref{eq:decomp}
\begin{equation}\label{eq:uedec2}
u_\e = (u_0 + u_\e^R) + u_{\e,-}^{BL} + u_{\e,+}^{BL} \;,
\end{equation}
where the smooth part $u_0 + u_\e^R$  
is bounded in $H^k(I)$ independently of $\e\in (0,1]$.

We consider now the case that $0 < b = const.$, 
i.e. the setting of Cor.~\ref{cor:bconst}. 
There, $u_{\e,\pm}^{BL}$ in \eqref{eq:uedec2}
are (translated and dilated) exponentials \eqref{eq:uBLexp}, 
and we may use Lemmas~\ref{lem:tanhnnexp}, \ref{lem:reluexp} 
to robustly express the exponential part of
$u_{\e,\pm}^{BL}$ in \eqref{eq:uedec2}.
%%%%%%%%%%%%%%%%%%%%%%%%% constant b %%%%%%%%%%%%%%%%%%%%%%%%%%%
\begin{theorem}\label{thm:sReLUbcst}
{[Robust deep $\ReLU$ NN expression]} 
\newline
Consider the BVP \eqref{eq:de}, \eqref{eq:bc}, 
with constant reaction coefficient $b(x) = b > 0 $,
and $f\in H^2(I)$, for $0<\e\le 1$.

Then for every emulation accuracy $0<\tau\leq 1$, 
and for every $0 < \e\leq 1$,
there exist deep feedforward $\ReLU$ NNs 
$\{\Phi^\tau_\e \}_{0<\e\leq 1}\subset \mathcal{NN}^{\ReLU}_{\bar{L},\bar{W},1,1}$
of
width and depth bounded for $\tau\downarrow 0$ 
(constants hidden in $\lesssim$ independent of $\e$)
as
\begin{equation}\label{eq:NNwddp}
\bar{W}(\tau) \lesssim |\log_2(\tau)|^2, 
\quad 
\bar{L}(\tau) \lesssim \tau^{-1/[2(1-\delta)]}
\;,
\end{equation}
where $0<\delta<1$ is arbitrary, fixed.
For such $\delta$
exist constants $C_\delta, \tilde{C}_\delta > 0$ 
such that 
for every $0 < \e, \tau \leq 1$
$$
M(\Phi^\tau_\e)  \leq C_\delta \tau^{-1/[2(1-\delta)]} 
\;,
$$
and
$$
\| u_\e - \realiz{(\Phi^\tau_\e}) \|_{1,\e,I} 
\leq 
\tau 
\leq 
\tilde{C}_\delta M^{-2(1-\delta)} 
\;,
\quad 
\| u_\e - \realiz{(\Phi^\tau_\e}) \|_{L^2(I)} 
\leq 
\tau^2
\leq 
\tilde{C}^2_\delta M^{-4(1-\delta)} 
\;.
$$
\end{theorem}
\begin{proof}
We recall the decomposition \eqref{eq:uedec2}. 
In \eqref{eq:uedec2}, 
we approximate the regular part $u_0+u_\e^R$
with the ReLU super-approximation result Proposition~\ref{prop:ReLUSupr},
where we fix the width $W$, and let $L\to \infty$.
The logarithmic terms in the bounds \eqref{eq:sReLUBd}, \eqref{eq:sReLUdw}
are absorbed into polynomial bounds via the (arbitrary small) parameter $\delta>0$.
The boundary layers in \eqref{eq:uedec2} are 
approximated $\e$-uniformly by ReLU NNs with Lemma~\ref{lem:reluexp}.

\noindent
Step 1: 
Approximation of the regular part.
The regular part $u_0+u_\e^R$ in \eqref{eq:uedec2} satisfies 
$\| u_0+u_\e^R \|_{H^2(I)} \leq C(b,f) < \infty$ for $0<\e \leq 1$.
Robust NN emulation rate bounds are
obtained by expressing $u_0 + u_\e^R$
by deep $\ReLU$-NNs with Prop.~\ref{prop:ReLUSupr}, 
with $q=2$, as follows: 
fixing $\bar{W}$, and increasing $\bar{L}$, 
\eqref{eq:sReLUBd} yields the existence of 
$\ReLU$-NNs 
$\{ \Phi_{\e}^{reg,M}\}_{\stackrel{M\in \mathbb{N}}{0<\e\leq 1}}$
of fixed widths $W(\Phi_{\e}^{reg,M}) \leq C^{reg}$
such that for any $0<\delta\leq 1$ 
there is a constant $C_\delta>0$ 
such that for NN sizes $M\in \N$ and every $0<\e \leq 1$ 
holds
\[
\begin{array}{rcl}
\| u_0+u_\e^R - \realiz{(\Phi_{\e}^{reg,M})} \|_{1,\e,I}
&\simeq&
\e | u_0+u_\e^R - \ \realiz{(\Phi_{\e}^{reg,M})} |_{H^1(I)}
\\
&+& 
 \| u_0+u_\e^R - \realiz{(\Phi_{\e}^{reg,M})} \|_{L^2(I)}
\\
& \leq & 
C_\delta M^{-(2-\delta)}(\e+ M^{-(2-\delta)}) \| u_0 + u_\e^R \|_{H^2(I)} 
\\
&\leq & C_\delta C(b,f) M^{-(2-\delta)}(\e+ M^{-(2-\delta)}) 
\\
&\leq & 2 C_\delta C(b,f) M^{-(2-\delta)}
\;.
\end{array}
\]

\noindent
Step 2: Approximation of the exponential boundary layers \eqref{BLs}.
We use the assumption that $b$ is constant.
The explicit form \eqref{eq:uBLexp} of $u^{BL}_{\e,\pm}$
and 
Lemma~\ref{lem:reluexp}, with bounds
\eqref{eq:ReLUBL1}, \eqref{eq:ReLUBL0},
yields the existence of $\ReLU$-NNs 
%$\{ \Phi^{{\rm exp},\kappa,p}_\e \}_{0<\e\leq 1}$
$\{ \Phi^{BL,\tau}_{\e,\pm} \}_{0<\e\leq 1}$
such that
\[
\| u^{BL}_{\e,\pm} - \realiz{(\Phi^{BL,\tau}_{\e,\pm})} \|_{1,\e,I} 
\lesssim e^{-bp} \stackrel{!}{\leq} \tau,
\]
and whose depth and size are bounded by \eqref{eq:LN} as
\[
L(\Phi^{BL,\tau}_{\e,\pm}) \lesssim p(1+\log_2(p)) \simeq |\log_2(\tau)|( 1 + \log_2(|\log_2(\tau)|)),
\]
and
%\[
%W(\Phi^{BL,\tau}_{\e,\pm}) \lesssim p \simeq |\log_2(\tau)| 
%\;.
%\]
%Using these bounds we observe that
%\[
%\#(\Phi^{BL}_{\e,\pm})
%\lesssim L(\Phi^{BL}_{\e,\pm}) W(\Phi^{BL}_{\e,\pm}) 
%\lesssim |\log_2(\tau)|^2 (1+\log_2(|\log_2(\tau)|) 
%\]
%and
\[ 
M(\Phi^{BL,\tau}_{\e,\pm})
%\lesssim L(\Phi^{BL}_{\e,\pm}) W(\Phi^{BL}_{\e,\pm})^2 
\lesssim |\log_2(\tau)|^2 %(1+\log_2(|\log_2(\tau)|)
\]
with constants hidden in $\lesssim$ independent of $\e \in (0,1]$.

\noindent
Step 3: 
Construction of the $\ReLU$-NNs 
$\{\Phi^\tau_\e \}_{0<\e\leq 1}$.
Given emulation accuracy $\tau \in (0,1]$, 
we choose $M(\tau)\simeq \tau^{-1/(2-2\delta)}$
and set, 
with $\Parallel{\Phi_1,\Phi_2}$ denoting the parallelization of NNs $\Phi_1,\Phi_2$,
\[
\Phi^\tau_\e :=
% \Phi^{BL,\tau}_{\e,\pm} =
 \Parallel{\Phi^{BL,\tau}_{\e,\pm},\Phi_{\e}^{reg,M}}
\;.
\]
Then 
$M(\Phi^{\tau}_\e) \leq 2 M(\Phi^{BL,\tau}_{\e,\pm}) + M(\Phi_{\e}^{reg,M})$,
giving the claimed NN size bound.
\end{proof}
\begin{remark}\label{rmk:sReLU}
Observe that the size of the NN for the emulation of $u^{BL}_{\e,\pm}$ 
in Step 2 is negligible when compared to the size of the 
$\ReLU$-NN in Step 1, due to the (uniform w.r. to $\e$)
exponential expressivity of 
$u^{BL}_{\e,\pm}$ by $\ReLU$-NNs furnished by Proposition~\ref{prop:ReLUSupr}.
\end{remark}
\begin{remark}\label{rmk:tanhNN}
Arguing with Lemma~\ref{lem:tanhnnexp},
the same result is true if a fixed number of 
$\tanh( )$ activations are included to express $u^{BL}_{\e,\pm}$ 
\emph{exactly} in the step 1 of the proof.
The corresponding NNs $\Phi^{\tau}_\e$ will have fixed width
and depth bounded as in \eqref{eq:NNwddp}.
\end{remark}
%
%%%%%%%%%%%%%%%%%%%%%%%%%%%%%%%%%%%%%%%%%%%%%%%%%%%%%%%%%55
\section{Conclusions and Generalizations}
\label{sec:Concl}
%%%%%%%%%%%%%%%%%%%%%%%%%%%%%%%%%%%%%%%%%%%%%%%%%%%%%%%%%%
For model, linear singular perturbation boundary value problems
\eqref{eq:de}, \eqref{eq:bc} 
in a bounded interval $I = (-1,1)$,
with data belonging to finite order Sobolev or Besov spaces, 
we
established robust expression rate bounds of 
the parametric solutions by ReLU NNs, whose realizations 
are CpwL functions of $x$. 
The presently proved, 
\emph{robust algebraic expression rate bounds} 
under, in a sense minimal, regularity on the data $b$ and $f$ of 
the problem \eqref{eq:de}, \eqref{eq:bc},
complement the \emph{robust exponential expression rate bounds}
that were shown in \cite{OSX24_1085}
under stronger, analytic regularity hypotheses on the data.

In Sec.~\ref{sec:FinReg}, we established the solution decomposition
\eqref{eq:decomp}, under the regularity hypotheses $b,f\in H^1(I)$.
While this is weaker than, e.g., the $C^2(\overline{I})$ regularity
often imposed in arguments relying on the maximum principle and/or
barrier functions, 
the present approach could be extended to
data regularity below $H^1(I) = W^{1,2}(I)$. 
To this end, one would need to adopt
in the proof of Theorem~\ref{prop:Reg1} 
a non-hilbertian setting, 
in (exponentially weighted versions of) spaces $W^{k,p}(I)$, $k=1,2$,
with $1\leq p <2$, 
requiring, in a sense minimally\footnote{Discontinuous $b(x)$ can give rise to interior layers \cite{RST2nd}.} 
in \eqref{eq:de} that 
$b,f \in W^{1,1}(I) \subset C^0(\overline{I})$.

Similar algebraic expression rate bounds (which are known)
were shown to hold for so-called $h$-version 
Finite Element approximations of fixed polynomial degree $p\geq 1$, 
on either
so-called Shishkin-meshes or on exponentially graded partitions of 
the domain $I$.
We refer to \cite{RST2nd} and in particular to the 
lucid discussion with references in \cite[Section~2]{Linss1985}. 
These results were shown here to imply corresponding 
robust, algebraic expression rate bounds for shallow ReLU NNs.
Corresponding robust expression rate bounds for spiking NNs are
implied via the ReLU-to-Spiking conversion as in \cite{OSX24_1085}.

Using the tools developed in \cite{li2025superapproximationratesreluneural}
we showed furthermore that for 
reaction coefficient $b$ and the source term $f$ in \eqref{eq:de}
belonging to $H^1(I)$, 
deep ReLU NNs furnish (essentially, up to logarithmic terms)
\emph{twice the rate of robust convergence in energy} 
that can be proved for traditional, linear approximations of FE or FV type, 
on the mentioned, locally refined, layer-resolving partitions. 

For constant reaction term, use of $\tanh( \cdot )$ or sigmoid 
activations allows to \emph{exactly} represent the (exponential) 
boundary layer functions $u^{BL}_{\e,\pm}$ in the decomposition \eqref{eq:decomp},
with a subnetwork of fixed size, in effect adding exponential boundary layers
to the NN feature space.
%and a Shishkin-type boundary layer resolving patch construction 
%as e.g. presented in \cite[Section~2]{Linss1985},
%we found that deep ReLU NNs can furnish a robust approximation rate
%which is (up to log-terms) \emph{twice the rate afforded by traditional 
%$h$-FE or $h$-FV  methods}, for given data regularity.

%\todo{[Attn: can also obtain for fixed polynomial degree $p\geq 2$ 
%       under suitable higher (still finite) regularity  
%       corresponding higher robust rates for ReLU NNs of fixed depth $L(p) > 2$]; possibly
%       in Conclusion.}{\color{magenta} [I agree it should appear in the Conclusions.]}
%
\bibliographystyle{abbrv}
\bibliography{references}
\end{document}